\documentclass[11pt,leqno]{article}

\usepackage{amsthm,amsfonts,amssymb,amsmath,oldgerm}
\usepackage{fullpage}
\usepackage{graphicx}

\usepackage{mathrsfs}
\usepackage{xcolor}
\usepackage[colorinlistoftodos]{todonotes} 

\usepackage[colorlinks=true]{hyperref}
\usepackage[margin=1in]{geometry}
\usepackage[margin=1in]{caption}

\numberwithin{equation}{section}

\def\be{\begin{equation}}
\def\ee{\end{equation}}
\def\bse{\begin{subequations}}
\def\ese{\end{subequations}}

\renewcommand\d{\partial}
\renewcommand\a{\alpha}

\def\eps {\varepsilon}

\newcommand{\R}{\mathbb R}
\newcommand{\C}{\mathbb C}

\newcommand{\RM}{{\mathbb{R}}}
\newcommand{\CM}{{\mathbb{C}}}
\newcommand{\NM}{{\mathbb{N}}}
\newcommand{\ZM}{{\mathbb{Z}}}
\newtheorem{theorem}{Theorem}[section]
\newtheorem{proposition}[theorem]{Proposition}
\newtheorem{corollary}[theorem]{Corollary}
\newtheorem{lemma}[theorem]{Lemma}
\newtheorem{remark}[theorem]{Remark}
\theoremstyle{definition}
\newtheorem{definition}[theorem]{Definition}
\newtheorem{assumption}[theorem]{Assumption}

\allowdisplaybreaks[3]

\title{On the Modulation of Wave Trains in the Ostrovsky Equation}

\author{Mathew~A.~Johnson\thanks{Department of Mathematics, University of Kansas, 1460 Jayhawk Boulevard, 
Lawrence, KS 66045, USA; matjohn@ku.edu}\quad
Jeffrey Oregero\thanks{Department of Mathematics, University of Kansas, 
1460 Jayhawk Boulevard, 
Lawrence, KS 66045, USA; oregero@ku.edu}\quad 
Wesley R. Perkins\thanks{Division of Mathematics and Computer Science, Lyon College, 2300 Highland Road, Batesville, AR 72501 USA;
wesley.perkins@lyon.edu}}

\date{\today}

\begin{document}
	
\maketitle

\begin{abstract}
We consider the nonlinear wave modulation of arbitrary amplitude periodic traveling wave
solutions of the Ostrovsky equation, which arises as a model for the unidirectional propagation of small-amplitude, weakly nonlinear
surface and internal gravity waves in a rotating fluid of finite depth.    While the modulation of such waves with
asymptotically small amplitudes of oscillation (the so-called Stokes waves) has been studied in several works,
our goal is to understand the modulational dynamics of general amplitude wave trains.  To this end, we first use
Whitham's theory of modulations to derive
a dispersionless system of quasilinear partial differential equations that is expected to model the slow evolution
of the fundamental characteristics of a given wave train.  In practice, the modulational stability or instability of a given wave train is 
considered to be determined by the hyperbolicity or ellipticity, respectively, of the resulting system of Whitham modulation equations.
Using rigorous spectral perturbation theory we then study the spectral (linearized) stability problem for a given wave train solution of the Ostrovsky equation, 
directly connecting the hyperbolicity or ellipticity of the associated Whitham system to the rigorous spectral stability problem for the underlying wave.
Specifically, we prove that strict hyperbolicity of the Whitham system implies spectral stability near the origin in the spectral plane, i.e. so-called
spectral modulational stability, while ellipticity implies spectral instability of the underlying wave train.  
\end{abstract}

\section{Introduction}
\label{s:intro}

The famous Korteweg-de Vries (KdV) equation 
\begin{equation}\label{e:kdv}
u_t+uu_x+\beta u_{xxx}=0,
\end{equation}
often arises as a canonical model for the unidirectional propagation of small amplitude, long wavelength internal
or surface water waves in the ocean. 
For the KdV equation, the parameter $\beta\neq 0$ is irrelevant and can be scaled to $\beta=1$.
While the KdV equation has been enormously successful in this context, accounting for the competition between the main effects
in the physics of nonlinear waves (namely, dispersion and nonlinearity), 
one drawback in the modeling is that it ignores the fact that such waves in the ocean necessarily occur in a rotating fluid.  Although
the effect of background rotation may arguably be considered irrelevant for waves on small space and time scales, the fact that 
surface and internal water waves can be observed to persist in the ocean for several days  strongly suggests there are situations when
the rotational effects of the fluid may become important.

To incorporate the effects of background rotation, Lev Ostrovsky proposed in 1978 the equation
\be 
(u_t + uu_x + \beta u_{xxx})_x = \gamma u
\label{e:ostrov}
\ee
with $\beta\neq 0$ as a model for the unidirectional propagation of small amplitude, long wavelength internal or surface water waves in a rotating
fluid \cite{Lev78}.  In this way, the Ostrovsky equation \eqref{e:ostrov}, as it is now known, may be considered as a rotational modification
of the KdV equation \eqref{e:kdv}.  The constant $\gamma>0$ is a parameter modeling the Coriolis dispersion due to the rotation of the background fluid \cite{GS91}.  
In contrast to the KdV equation \eqref{e:kdv}, the non-zero constant $\beta$ can not be scaled out and instead
determines the type of dispersion, with $\beta>0$ applying to 
surface and internal ocean waves and $\beta<0$ applying to surface capillary waves and to the modeling of oblique magneto-acoustic waves in plasma. 
Noting that the linear dispersion relation for \eqref{e:ostrov} is given by
\[
\omega(k) = \frac{\gamma}{k}-\beta k^3,
\]
we see that the additional $\gamma u$ term in \eqref{e:ostrov} significantly effects the dispersion of long waves, and hence
in this way introduces a ``large-scale" dispersion into the modeling that is not present in KdV based theories.  
With the incorporation of large scale dispersion, since its introduction the Ostrovsky equation has become a popular model in oceanology and
has found relevance in several other fields such as plasma physics.  
A more extensive account of the applications and history 
of the Ostrovsky equation is beyond the scope of this work, and interested
readers may consult, for example, \cite{B92,OS98,Step20} and references therein.

Given the physical relevance of \eqref{e:ostrov} and its close relationship to the KdV equation, it is clearly  interesting and meaningful to study
the differences in the dynamics of solutions between the KdV and Ostrovsky models.  
One significant way in which the dynamics of solutions differs between the KdV and Ostrovsky equations has to do with the stability
of periodic traveling waves (sometimes referred to as wave trains).  In the context of the full Euler equations, Benjamin \& Feir and (independently) Whitham used formal
asymptotic methods to argue that periodic traveling  water waves with asymptotically small amplitude of oscillation, the so-called Stokes waves, 
exhibit a modulational / side-band instability whenever the frequency is sufficiently large (i.e. their wavelength is sufficiently short).  A rigorous
proof of this so-called Benjamin-Feir instability, sometimes referred to as a modulational instability, 
was first provided by Bridges \& Mielke \cite{BM95}.  While all periodic traveling wave solutions
of the KdV equation are known to be dynamically (spectrally, modulationally, and orbitally) stable, the Stokes waves for the Ostrovsky equation
do in fact exhibit a modulational instability for short waves.
This was investigated for the Ostrovsky equation formally via weakly nonlinear analysis in \cite{GH08,GSA16}, and
just very recently the modulational instability of short waves in \eqref{e:ostrov} was rigorously verified via a delicate study of the spectrum
of the associated linearized operator 
in \cite{BJP24}.  

\

The primary goal of this work is to investigate
the modulational stability and instability of periodic traveling wave solutions in the Ostrovsky equation outside of the asymptotically small amplitude (i.e. Stokes) regime, thereby
providing a theory that applies to general amplitude wave trains and their modulations in \eqref{e:ostrov}.

\

While the modulation theory of Stokes waves is relatively well developed, the theory of wave modulation for general amplitude periodic traveling waves 
is considerably less developed.  There are effectively two approaches. The first is to apply Whitham's theory of wave modulations, which can 
be executed via a formal multiple scales / WKB asymptotic calculation where one separates in space-time a fast scale associated with the periodic
wave motion from a slow scale associated with the slow variation of the fundamental wave characteristics (e.g. the amplitude, frequency, momentum, etc.).  
The result at leading order is a dispersonless system of quasilinear partial differential equations (PDEs) 
that formally govern the slow space-time evolution of the fundamental wave characteristics.  Whitham's theory
posits that the well-posedness (i.e. hyperbolicity) of this system implies the modulational stability of the underlying wave, while the ill-posedness (i.e. ellipticity) 
implies the modulational instability of the underlying wave.  Whitham's theory is by now well-developed, and various researchers have used
Whitham's theory of modulations to describe finer dynamical behaviors within a given system.  For example, Whitham's theory of modulations has
proven to be a key tool in dispersive shock wave theory and the theory of soliton gases.   See, for instance, \cite{El21,EH16} and references therein.  
Furthermore, Whitham's theory opens itself to direct numerical
investigation, as the resulting quasilinear system is typically of low dimension.

In the context of the Ostrovsky equation \eqref{e:ostrov}, we will show that the set of all periodic traveling wave solutions constitute a
two-dimensional smooth manifold parameterized by the spatial frequency $k$ and the momentum $P$ (a conserved quantity, as we will see) of the wave.
Using a multiple-scales / WKB approximation, we will then show that slow modulations of 
a given periodic wave train $\phi_0(x)=\phi_0(x;k_0,P_0)$ will (formally and approximately) evolve slowly as a modulated wave train of the form
\[
u(x,t;X,S)\approx \phi(x;k(X,S),P(X,S)) ,
\]
where $(X,S)=(\eps x,\eps t)$ are the slow space-time variables and where the modulation functions $k$ and $P$ evolve (slowly) near $(k_0,P_0)$ 
according to a dispersionless quasilinear first order system of the form
\begin{equation}\label{e:W_intro}
\partial_S(k,P) = \mathcal{W}(\phi)\partial_X(k,P),
\end{equation}
where $\mathcal{W}(\phi)$ is a $2\times 2$ matrix depending on the underlying slowly modulated wave $\phi$.  According to Whitham's theory
of modulations, the stability of the base wave $\phi_0$ to slow modulations is determined by the eigenvalues of the constant matrix $\mathcal{W}(\phi_0)$.
Indeed, linearizing the Whitham system \eqref{e:W_intro} about $\phi_0$, here treated as a constant function of the slow variables $(X,S)$, we see the eigenvalues 
of the linearization of \eqref{e:W_intro} will be of the form
\[
\nu_j(\xi)=i\xi\alpha_j,~~j=1,2
\]
where the $\alpha_j$ are the eigenvalues of the matrix $\mathcal{W}(\phi_0)$ and $\xi\in\RM$ is a parameter.  The hyperbolicity of the Whitham system \eqref{e:W_intro} at $\phi_0$,
meaning all the $\alpha_j$ are real, indicates a (marginal) spectral stability of $\phi_0$ as a solution of \eqref{e:W_intro}.  On the other hand, the ellipticity of \eqref{e:W_intro}
at $\phi_0$, meaning the $\alpha_j$ have non-zero imaginary part, imply the spectral instability of $\phi_0$ as a solution of \eqref{e:W_intro}, indicating that the underlying
periodic traveling wave $\phi_0$ is unstable to slow modulations.  

\begin{remark}
The application of Whitham's theory to the Ostrovsky equation as presented here appears to be new in the literature.  In \cite{WJ17}, Whitfield \& Johnson
apply Whitham's modulation theory in a particular scaling of the weak-rotation limit $\gamma\to 0^+$.  In this singular KdV type limit, the resulting modulation
system derived in \cite{WJ17} is $3$-dimensional, owing to the fact that in the limiting $\gamma=0$ model (i.e. the KdV equation) the manifold of periodic traveling
wave solutions is three-dimensional.  See also Remark \ref{r:WJ} below.  One of the main contributions of our work is to apply Whithams' theory of modulations to the Ostrovsky equation
for a general (fixed) value of $\gamma>0$.  
\end{remark}

A second approach to studying the modulational stability/instability for a given PDE is to study the spectrum of the associated 
linearization to localized perturbations in a sufficiently small neighborhood of the origin in the spectral plane.  
Unlike Whitham's theory described above, which is based on formal asymptotic methods, this approach is often carried
out either with rigorous functional analytic techniques or by well-conditioned numerical methods.  In the context of the
Ostrovsky equation, after moving to a co-moving frame the spectral stability of a given traveling wave solution $\phi_0$
is governed by the spectrum of a linear operator pencil\footnote{Sometimes referred to as a ``generalized spectral problem."  
In the literature, however, spectral problems of the form $(\lambda B-A)v=0$ are often referred to as linear operator pencils.} of the form
\begin{equation}\label{e:lin_intro}
\mathcal{L}[\phi_0]v=\lambda \partial_\theta v
\end{equation}
considered for $v\in L^2(\RM)$, where here $\lambda\in\CM$ is the spectral (i.e. linearized growth) parameter and 
where $\mathcal{L}[\phi_0]$ is a non-symmetric, fourth-order linear differential operator with periodic coefficients depending on the 
profile $\phi_0$.  Due to the non-invertibility of the operator $\partial_\theta$ on $L^2(\RM)$, even in the presence of a mass-zero constraint,
the above linear operator pencil
can not be recast as a (standard) spectral problem for an integro-differential operator by inverting $\partial_\theta$.
Handling this analytical detail, and in particular the spectral perturbation theory for linear operator pencils, is a key feature of our analysis.

\begin{remark}
There are several works in the literature where the spectral stability for a given nonlinear wave $\phi_0$ is determined
by a  linear operator pencil
of the form \eqref{e:lin_intro}.  See, for example, \cite{HSS12_2,KS14,PS20,SS16_2}.  These previous studies, however,
concerned either the  case where $\phi_0\in L^2(\RM)$ is a solitary wave with \eqref{e:lin_intro} posed on $L^2(\RM)$, or the case where $\phi_0\in L^2_{\rm per}(0,T)$
is a $T$-periodic wave with \eqref{e:lin_intro} posed on $L^2_{\rm per}(0,T)$.  In the solitary wave case, the linearized operator can be decomposed via
\[
\mathcal{L}[\phi_0]=\mathcal{L}_0+\mathcal{K},
\]
where here $\mathcal{L}_0$ is a (typically positive-definite) constant coefficient pseudo-differential operator and $\mathcal{K}$ is a relatively compact perturbation.  This
makes the determination of the essential spectrum trivial (via the Weyl Essential Spectrum Theorem), thus reducing the problem to that of studying the presence of
unstable eigenvalues.  In the periodic case, where \eqref{e:lin_intro} is considered with periodic boundary conditions, there is in fact no essential spectrum and the stability
problem again reduces to the study of eigenvalues.  As we will see, in the present work we study \eqref{e:lin_intro} when $\phi_0$ is $T$-periodic but where the 
spectral problem \eqref{e:lin_intro} is considered on $L^2(\RM)$.  By standard results in Floquet-Bloch theory the spectrum of \eqref{e:lin_intro} in this case is comprised entirely
of essential spectrum, and hence the previous studies of spectral problems of the form \eqref{e:lin_intro} do not apply.
\end{remark}

Through the application of rigorous, functional analytic techniques we will study the $L^2(\RM)$-spectrum associated 
with the linear operator pencil \eqref{e:lin_intro} in a neighborhood of the origin $\lambda=0$ in the spectral plane when $\phi_0$ is a periodic
traveling wave solution of the Ostrovsky equation \eqref{e:ostrov}.  In particular,
we prove that the $L^2(\RM)$-spectrum near the origin consists of two $C^1$ curves (of essential spectrum) that admit the asymptotic expansion
\[
\lambda_j(\xi)=ic\xi+\nu_j(\xi)+o(\xi)=i\xi(\alpha_j+c)+o(\xi),~~|\xi|\ll 1,~~j=1,2
\]
where the $\alpha_j\in\CM$ are precisely the eigenvalues of the matrix $\mathcal{W}(\phi_0)$ coming from Whitham's theory of modulations.\footnote{In
the case where the $\alpha_j$ are distinct, the $o(\xi)$ error  can be replaced by $\mathcal{O}(\xi^2)$.}  Thus, not only do we perform
a rigorous spectral stability study, but we are able to rigorously justify the stability predictions coming from Whithams' theory.  This rigorous
connection is the content of our main result.

\begin{theorem}\label{t:main}
Suppose that $\phi_0$ is a $T_0=1/k_0$ periodic traveling wave solution of \eqref{e:ostrov} with wave speed $c_0>0$, and suppose that the set of nearby periodic
traveling wave profiles $\phi$ with speed close to $\phi_0$ is a two-dimensional smooth manifold parameterized by $(k,P(\phi))$, where $1/k$ denotes the fundamental
period of the wave and $P$ denotes the momentum of the wave.  If the Whitham system is strictly hyperbolic at $(k_0,P(\phi_0))$, then the wave $\phi_0$ is spectrally
modulationally stable, i.e. the $L^2(\RM)$-spectrum associated with the linearization of \eqref{e:ostrov} about $\phi_0$ is purely imaginary in a sufficiently small neighborhood
of the origin in the spectral plane.  Furthermore, a sufficient condition for $\phi_0$ to be spectrally unstable to perturbations in $L^2(\RM)$ is that the Whitham modulation
system is elliptic at $(k_0,P(\phi_0))$.
\end{theorem}

\begin{remark}\label{r:stokes}
In Section \ref{s:stokes} below, we show that the Whitham modulation system \eqref{e:W_intro} evaluated at a Stokes wave  for the Ostrovsky equation is strictly hyperbolic 
for sufficiently long waves (i.e. when the spatial frequency $k$ is sufficiently small) while it is elliptic  for sufficiently short waves: see Remark \ref{r:stokes2} below.
This demonstrates that the cases of strict hyperbolicity as well as ellipticity of the Whitham modulation system both do occur in the
case of the Ostrovsky equation \eqref{e:ostrov}.  By Theorem \ref{t:main} it follows that some periodic traveling wave solutions
of \eqref{e:ostrov} will be spectrally unstable to perturbations in $L^2(\RM)$ while others will be exhibit a (spectral) modulational stability.  
In forthcoming work, the authors will provide a thorough numerical investigation of the hyperbolicity/ellipticity of the Whitham modulation
system for the Ostrovsky equation.  These numerical calculations will directly complement the analytical investigation here by providing, through Theorem \ref{t:main},
rigorous (spectral) modulational stability/instability results from well-conditioned numerical studies of the Whitham
system \eqref{e:W_intro}.
\end{remark}

Importantly, our work provides both a derivation of Whitham's modulation
system for the Ostrovsky equation, for general amplitude periodic traveling waves and for general Coriolis parameters $\gamma>0$, 
as well as a mathematically rigorous investigation of the spectral stability problem about a given periodic traveling wave.
Morevoer, the above theorem  can be used in two different ways.  First, it allows one to use either rigorous mathematics or well-conditioned numerical techniques to study the
Whitham system \eqref{e:W_intro} and, from that, infer mathematically rigorous results pertaining to the modulational stability of the underlying wave train.
On the other hand, one can use Floquet-Bloch / Galerkin methods to numerically study the spectrum of the linearization in a neighborhood of the spectral
plane, and from these numerical computations be able to make direct inferences regarding the dynamics of solutions of the associated Whitham modulation
system (this latter approach may be useful in the study of dispersive shock waves).  


Finally, although it is not the subject of our investigation here, a special case of the Ostrovsky equation \eqref{e:ostrov} occurs
when the third-order dispersion term is neglected, corresponding to $\beta=0$.  In this case, \eqref{e:ostrov} reduces to
\be
\label{e:reduced_ostrov}
(u_t + uu_x)_x = \gamma u,
\ee
which is often referred to as the reduced Ostrovsky equation or, sometimes, as the Hunter-Saxton model.  
Compared to the classical Ostrovsky equation \eqref{e:ostrov}, the reduced model \eqref{e:reduced_ostrov} has several mathematical advantages.
For example, the reduced Ostrovsky equation is known to possess a Lax pair, and is hence completely integrable via the inverse scattering transform
while also possessing an infinite number of conserved quantities.  When $\beta\neq 0$, however, the Ostrovsky equation \eqref{e:ostrov}
is not completely integrable (although it still does admit a Hamiltonian structure).  While all of our forthcoming analysis
will be carried out in the $\beta\neq 0$ case, we will at the end of the manuscript make a series of comments regarding the applicability of our
work to this reduced $\beta=0$ case.

\

The organization of this paper is as follows.  We emphasize that throughout our work we will consider the $\beta\neq 0$ case.  
In Section \ref{s:basics} we discuss some of the basic properties of the Ostrovsky equation \eqref{e:ostrov}, 
including the Hamiltonian structure and conserved quantities as well as the existence of periodic traveling wave solutions.  In particular,
we will introduce our two main structural assumptions used throughout this work: the first having to do with the dimensionality and parameterization of the manifold of periodic
traveling wave profiles, and the second corresponding to the non-degeneracy of the linearization arising from the Hamiltonian structure.  Both of these assumptions
are supported by various existing numerical and analytical works in the literature.  In Section \ref{s:Whitham} we derive the Whitham modulation system
associated to the Ostrovsky equation \eqref{e:ostrov} via a direct multiple-scales / WKB approximation.  We will further provide an explicit study of the Whitham 
system in the case of the Stokes waves for the Ostrovsky equation, and will connect the Whitham system to the well-known Lighthill condition for the modulational
instability of Stokes waves.  

In Section \ref{s:modtime} we perform a rigorous spectral stability calculation, using Floquet-Bloch theory and spectral
perturbation theory to derive a $2\times 2$ matrix whose eigenvalues rigorously encode the $L^2(\RM)$-spectrum in a sufficiently small neighborhood of the 
origin in the spectral plane for the associated linearized problem.  We then provide the proof of Theorem \ref{t:main} in Section \ref{s:th_proof}, providing
a term-by-term comparison of the Whitham matrix $\mathcal{W}$ discussed above and the $2\times 2$ matrix coming from our rigorous theory.  
We note that while the spectral and spectral perturbation theory for linear operator pencils is similar to the classical theory
for closed, densely defined linear operators (see Kato's  book \cite{K76}), there are indeed some important technical differences.  As such,
in an attempt to make our work as self contained as possible we collect and present the relevant results in Appendix \ref{s:pencil}.

Finally, while all of the aforementioned analysis is conducted entirely in the $\beta\neq 0$ case, we close in Section \ref{s:W_reducedostrov} by providing
a discussion about the applicability of our results to the reduced Ostrovsky equation \eqref{e:reduced_ostrov}.  Assuming that our main structural
conditions continue to hold in this $\beta=0$ case, which does seem to be supported by other works in the literature, our analysis
suggests that all of our calculations and our main result Theorem \ref{t:main} continue to hold in the reduced case.  Interestingly,
our work suggests that although the reduced $\beta=0$ model is clearly a singular perturbation of the classical Ostrovsky equation \eqref{e:ostrov},
the modulation system for the reduced Ostrovsky equation is a \emph{regular perturbation} of that associated with \eqref{e:ostrov}.  

\

\noindent
{\bf Acknowledgments:} The work of MJ was partially supported by the NSF under grant DMS-2108749.  The authors are thankful
to Yuri Latushkin and Fritz Gesztesy for helpful conversations regarding the spectral theory for linear operator pencils, in particular
for making us aware of the works \cite{Baum85,GGK90,MP15}.

\

\noindent
{\bf Notation:} 
Let $X$ be a (complex) Banach space.  By a linear operator pencil on $X$ we mean a linear operator
of  the form $\lambda B- A$ where $A,B:X\to X$ are closed, densely defined linear operators\footnote{Clearly we must assume $D(A)\cap D(B)\neq\emptyset$.}
and $\lambda\in\CM$ is a complex spectral parameter.  The resolvent set $\rho(A,B)$ of the pencil $\lambda B-A$ is the 
set of all $\lambda\in\CM$ such that $\lambda B-A$ is boundedly invertible, and the spectrum of the pencil\footnote{When appropriate,
we will emphasize the underlying Banach space $X$ by writing $\sigma_X(A,B)$ and $\rho_X(A,B)$.} is $\sigma(A,B)=\CM\setminus\rho(A,B)$.
The spectrum of a given pencil may, in the usual way be decomposed into point spectrum $\sigma_p(A,B)$, consisting of isolated eigenvalues,
and the essential spectrum\footnote{In particular, $\lambda\in\sigma_{\rm ess}(A,B)$ provided $\lambda B-A$ is not Fredholm or
is Fredholm but not index zero.  Further, throughout if $B=I$ then we use the standard notation $\rho(A,B)=\rho(A)$ and $\sigma(A,B)=\sigma(A)$.}  $\sigma_{\rm ess}(A,B)=\sigma(A,B)\setminus\sigma_p(A,B)$.  Further, if $\lambda_0\in\sigma_p(A,B)$ has eigenfunction
$v\in X\setminus\{0\}$, the function $w\in X\setminus\{0\}$ is a generalized eigenfunction above $v$ (i.e. it is part of a Jordan chain) provided that
\[
(\lambda_0 B-A)w=Bv.
\]
The geometric multiplicity of an eigenvalue $\lambda_0\in\sigma_p(A,B)$ is defined as the number of corresponding linearly independent eigenfunctions,
i.e. it is the dimension of ${\rm ker}\left(\lambda_0 B-A\right)$, while the algebraic multiplicity is the greatest value of the sum of the lengths of the Jordan chains
corresponding to linearly independent eigenvectors associated to $\lambda_0$.  For more information and background on linear operator pencils,
see, for example, \cite{GGK90,MP15}.

Also, throughout our work, if $f\in L^2_{\rm per}(0,T)$ has mean-zero, i.e. 
$\int_0^T f~dx=0$, then $\partial_x^{-1} f$ denotes the (unique) mean-zero antiderivative of $f$.

\section{Basic Properties of the Ostrovsky Equation}\label{s:basics}

In this section, we begin by discussing the conservation laws and conserved quantities associated with solutions
of the Ostrovsky equation \eqref{e:ostrov}.  We then discuss the existence theory for periodic traveling waves, and
introduce our main assumption regarding the local parametrization of the manifold of all such solutions.  Additionally,
we will address the non-degeneracy of the linearization associated with the Hamiltonian structure, which will play 
a key role in our forthcoming analysis.

\

Throughout this section, unless otherwise explicitly stated, we will assume that $\beta\neq 0$, i.e. our focus is on the general model \eqref{e:ostrov} 
and will \emph{not} be on the reduced model \eqref{e:reduced_ostrov}.

\subsection{Conservation Laws \& Conserved Quantities}

We begin by noting that the Ostrovsky equation is known \cite{T_09,YLHD18} to be locally well-posed in $H^s(\RM)$ for $s\geq -3/4$, and ill-posed
for $s<-3/4$.  Further, it is clear that if $u=u(x,t)$ is a solution of \eqref{e:ostrov} which is $T$-periodic in $x$, 
then $u$ is necessarily mean zero, i.e. we have
\begin{equation}\label{e:mean_zero}
\int_0^T u(x,t)dx=0
\end{equation}
for so long as the solution exists.  As such, it is clear that the operator $\partial_x^{-1}$, yielding the mean-zero antiderivative,
is well-defined on all such solutions.  With this in mind, we note that \eqref{e:ostrov} is known to admit the following
three conservation laws:
\begin{equation}\label{e:cons}
\left\{\begin{aligned}
	&u_t = \left(\frac{1}{\gamma}\left(u_t+\beta u_{xxx}+uu_x\right)_t\right)_x,~~~~
\left(\frac{u^2}{2}\right)_t=\left(\beta\left(\frac{1}{2}u_x^2-uu_{xx}\right)-\frac{1}{3}u^3+\frac{\gamma}{2}\left(\partial_x^{-1}u\right)^2\right)_x,\\
	&\left(\frac{u^3}{6}+\frac{\gamma}{2}\left(\partial_x^{-1}u\right)^2-\frac{\beta}{2}u_x^2\right)_t=\\
	&\qquad
		\left(\beta^2\left(u_xu_{xxx}-\frac{1}{2}u_{xx}^2\right)+\beta\left(uu_x^2-\frac{1}{2}u^2u_{xx}\right)-\beta\gamma u_x\left(\partial_x^{-1}u\right)
			+\frac{\gamma^2}{2}\left(\partial_x^{-2}u\right)^2-\frac{u^4}{8}\right)_x.
\end{aligned}\right.
\end{equation}
Upon integration, the first conservation law yields conservation of mass which here is equivalent to \eqref{e:mean_zero}.  
For spatially $T$-periodic functions as discussed above, the above also yields the conserved quantities
\begin{equation}\label{e:P}
P(u)=\frac{1}{2}\int_0^Tu^2dx~~{\rm and}~~H(u)=\int_0^T\left(\frac{u^3}{6}+\frac{\gamma}{2}\left(\partial_x^{-1}u\right)^2-\frac{\beta}{2}u_x^2\right)dx
\end{equation}
corresponding to the momentum and Hamiltonian (energy), respectively, of the wave.

In passing, we note also that \eqref{e:ostrov} is a Hamiltonian system and can be written in the form
\[
u_t=-\partial_x\frac{\partial H}{\partial u}(u),
\]
where the symbol $\partial/\partial u$ denotes the variational derivative: see, for example \cite{CIL07}.  In this context, one sees
that the integral \eqref{e:mean_zero} corresponds to a Casimir for the flow,\footnote{Being associated
with the kernel of the symplectic form $\partial_x$.} while the momentum $P$ is associated
with the fact that \eqref{e:ostrov} is invariant with respect to spatial translations.  We will comment more about the Hamiltonian structure
in the next section.

\subsection{Existence \& Parameterization of Periodic Traveling Waves}\label{s:existence} 

The existence of localized and periodic traveling wave solutions of the Ostrovsky equation \eqref{e:ostrov} has been considered
by several authors \cite{BJP24,BP23,HSS17,IM08,JP16,LL06,LL07,LV04,LLW12,PS20}.  Oftentimes, waves are either constructed using variational methods or techniques
from bifurcation theory.  In particular, we note that closed form solutions can not be constructed in terms of elementary
 functions (for example, in terms of hyperbolic or elliptic functions), and existence itself can not be determined easily 
from phase plane analysis (as we will see, the phase space is four-dimensional).  Taking motivation from
several works, however, below we introduce a natural structural assumption regarding the manifold of smooth periodic traveling
wave solutions to \eqref{e:ostrov}.  

Traveling wave solutions of \eqref{e:ostrov} correspond to solutions of the form
\[
u(x,t) = \phi(x-ct),
\]
where here $\phi$ is the wave profile and $c\in\RM$ is the wave speed.  Introducing the traveling coordinate $\eta=x-ct$,
it follows that the profile $\phi(\eta)$ is a stationary solution of the evolutionary equation
\[
\left(u_t-cu_\eta+uu_\eta+\beta u_{\eta\eta\eta}\right)_\eta=\gamma u.
\]
In particular, the profile $\phi$ is necessarily a solution of the ordinary differential equation (ODE)
\begin{equation}\label{e:profile}
\left(-c\phi'+\phi\phi'+\beta\phi'''\right)'=\gamma \phi,
\end{equation}
where here $'$ denotes differentiation with respect to the traveling variable $\eta$.  Since the profile ODE \eqref{e:profile} is fourth-order and can not be integrated, studying
the set of all such traveling wave profiles is extremely challenging.  Here, we briefly review three sets of existence results appearing in the literature:
two based on bifurcation theory and applied to particular asymptotic limits for the model, and a third result coming from variational methods.
From these results, we will make a natural assumption concerning the structure of the manifold of periodic solutions
to the profile equation \eqref{e:profile}.

The first existence result discussed here pertains to the weak-rotation limit, i.e. it establishes the existence of (general amplitude) periodic solutions
of \eqref{e:profile} in the limit $\gamma\to 0^+$.  Note that when $\gamma=0$ the profile equation can be integrated twice to give
\begin{equation}\label{e:quad_kdv1}
-c\phi+\frac{1}{2}\phi^2+\beta\phi''=a+b\eta
\end{equation}
for some constants of integration $a,b\in\RM$.  Seeking solutions which are spatially periodic (hence bounded on the whole line) forces the condition $b=0$,
in which case the integrated profile equation \eqref{e:quad_kdv1} is precisely the profile equation associated with the well-studied KdV equation.
In particular, \eqref{e:quad_kdv1} can be directly integrated to quadrature and shown the set of periodic solutions constitute a 
%
%
three-parameter family $\phi_{KdV}(\cdot;k,M,P)$, up to spatial translations, that can be smoothly
parameterized by the spatial frequency $k$, the momentum $P=\frac{1}{2}\int_0^{1/k}\phi_{KdV}^2 d\eta$, and the mass $M=\int_0^{1/k}\phi_{KdV} d\eta$.
In \cite{HSS17}, it was shown that a necessary and sufficient condition for the KdV wave $\phi_{KdV}$ to bifurcate
for $0<\gamma\ll 1$ to a solution of the profile equation \eqref{e:profile} is that the KdV wave has zero mass, i.e. that
\[
\int_0^{1/k}\phi_{KdV}d\eta=0
\]
in which case \eqref{e:profile} admits a $1/k$-periodic solution of the form
\[
\phi(\eta;k,P)=\phi_{KdV}(\eta;k,0,P)+\mathcal{O}(\gamma).
\]
This produces a two-parameter family of periodic traveling wave solutions to the KdV equation which continue for $0<\gamma\ll 1$
to solutions of \eqref{e:profile}, thus yielding a two-parameter family of mean-zero ``nearly-KdV" periodic traveling wave solutions of the Ostrovsky equation
parameterized by the spatial frequency $k$ and the momentum $P$.  

The second existence result discussed here concerns the existence of so-called Stokes waves, i.e. periodic solutions of \eqref{e:profile} with asymptotically
small amplitude of oscillation.  When seeking $T=1/k$-periodic solutions of \eqref{e:profile} it is convenient to rescale the spatial
variable as $\theta=k\eta$ so that a $T=1/k$-periodic solution of \eqref{e:profile} corresponds to a $1$-periodic solution of the rescaled
profile equation
\begin{equation}\label{e:profile_stokes}
k^2\left(-c\phi'+\phi\phi'+\beta k^2\phi'''\right)'=\gamma\phi.
\end{equation}
In the works \cite{BJP24,BP23}, the authors use an elementary Lyapunov-Schmidt argument to show that for each fixed $k>0$ there exists a one-parameter family of 
smooth, even, $1$-periodic solutions $(\phi(\cdot;a,k),c(a,k))$ of \eqref{e:profile_stokes} parameterized by $a\in\RM$ for $|a|\ll 1$.  
Further, the profile $\phi$ and the wave speed $c$ both
depend analytically on both $a$ and $k$ and admit the asymptotic expansion\footnote{Technically, for $\beta<0$ one must enforce the non-degeneracy condition
$(2\pi k)^4\neq -\gamma/n^2\beta$ for any $n\in\NM$ to ensure simplicity of the kernel of the linearization of \eqref{e:profile_stokes} about the zero solution.  At the level
of the asymptotic expansions in \eqref{e:stokes_waves} one can already see this arise (for $n=2$ at least) by ensuring that the denominator of $A_2(k)$ is non-zero.  Larger
values of $n$, likewise, ensure that later terms in the asymptotic expansions are well-defined.  Note there
is no such restriction when $\beta>0$.  For more details, see \cite{BJP24,BP23}.  In our work, we will not reference again this non-degeneracy condition and it will
just simply be assumed to hold when appropriate.}
\begin{equation}\label{e:stokes_waves}
\left\{
\begin{aligned}
\phi(\theta;a,k)&=a\cos(2\pi \theta)+2a^2A_2(k)\cos(4\pi \theta)+\mathcal{O}(a^3),\\
c(a,k)&=\frac{\gamma}{4\pi^2 k^2}- 4\beta\pi^2 k^2+a^2A_2(k)+\mathcal{O}(a^4),
\end{aligned}\right.
\end{equation}
as $a\to 0$, where 
\[
A_2(k)=\frac{2k^2\pi^2}{3\left(\gamma+64\beta k^4\pi^4\right)}.
\]
In this asymptotic regime, the constructed waves naturally induce a two-parameter family (up to spatial translations) of periodic
solutions to the profile equation \eqref{e:profile} that are smoothly parameterized 
by the spatial frequency $k$ and the amplitude $a$.  Moreover, calculating the momentum of the above Stokes waves to be
\[
P\left(\phi(\cdot;a,k)\right)=\frac{1}{2}\int_0^{1}\phi(\theta;a,k)^2d\theta=\frac{a^2}{4}+\mathcal{O}(a^3),
\]
it follows that such waves with $0<|a|\ll 1$ can, similarly to those constructed in the $\gamma\to 0^+$ limit above, be smoothly reparameterized 
by the frequency $k$ and the momentum $P$.  

Outside of the above distinguished asymptotic limits there has been relatively little work regarding the existence of periodic
solutions of \eqref{e:profile}.  We do note that in \cite{IM08} the authors use variational methods to construct a two-parameter family of periodic
traveling waves parameterized by the spatial frequency and the momentum.  While the dependence on the solution, along with the various associated Lagrange multipliers,
is not analyzed there this result provides further evidence that periodic solutions to \eqref{e:profile} generally form a 2-parameter family
of smooth solutions that  can be parameterized by the spatial frequency $k$ and the momentum $P$.

\

Based on the above survey of known existence results, moving forward we  assume that the set of periodic
solutions to \eqref{e:profile} constitutes a two-dimensional smooth manifold that can be (locally) parameterized by
the frequency $k$ and the momentum $P$ of the wave.  Specifically, we make the following structural assumption throughout our work.

\begin{assumption}[Existence of Periodic Traveling Waves]\label{a:soln}
There exists an open set $\mathfrak{B}\subset\RM_+\times\RM_+$ such that for any $(k,P)\in\mathfrak{B}$ there exists
a wave speed $c(k,P)$ and a non-constant, smooth, even, $1/k$-periodic function $\phi$ with momentum $P=\frac{1}{2}\int_0^{1/k}\phi^2d\eta$ 
such that $\phi$ solves the profile equation \eqref{e:profile} with $c=c(k,P)$.  Further, the profile $\phi(\cdot;k,P)$
and the wave speed $c(k,P)$ are smooth functions of $(k,P)\in\mathfrak{B}$.
\end{assumption}
%

\begin{remark}\label{r:exist_assumption}
While we attempted above to justify the legitimacy of Assumption \ref{a:soln} by surveying known results in the literature,
we will also see in Section \ref{s:whitham_derivation} below that Assumption \ref{a:soln} is effectively 
\emph{required} in order for the Whitham modulation system to describe an evolutionary
equation on the manifold of periodic traveling wave solutions to the Ostrovsky equation.  That is, if Whitham's theory of modulations has a hope
of describing the slow modulation of periodic wave trains of the Ostrovsky equation \eqref{e:ostrov} then Assumption \ref{a:soln} must
hold.  We will revisit this in Remark \ref{r:re-param} below.
\end{remark}

\subsection{Nondegeneracy of the Linearization}

In this section, we introduce a final assumption that is key to our forthcoming work.  As motivation, we observe
that the profile equation \eqref{e:profile} is precisely the Euler-Lagrange equation for 
the action functional
\[
\Lambda(u):=H(u)-cP(u)
\]
acting on the space of mean-zero functions in $L^2_{\rm per}(0,T)$, where here $H$ and $P$ are the Hamiltonian (energy) and momentum  functionals
defined in \eqref{e:P}.  It follows that we can write \eqref{e:ostrov} in traveling coordinates in the Hamiltonian form
\[
u_t=-\partial_\eta\frac{\partial\Lambda}{\partial u}(u)	.
\]
As is common in Hamiltonian problems, the stability of a given equilibrium solution $\phi$ depends strongly on the spectral properties of the (symmetric)
Hessian operator
\[
\widetilde{\mathcal{L}}[\phi]:=\frac{\delta^2\Lambda}{\delta u^2}(\phi) 
=-\gamma\partial_\eta^{-2}-c+\phi+\beta\partial_\eta^2
\]
here considered as acting on the mean-zero subspace of $L^2_{\rm per}(0,T)$.  Of particular concern is the so-called non-degeneracy of the Hessian, corresponding
to the condition that the kernel of the Hessian is generated precisely by the continuous Lie-point symmetries of the underlying model.  In the case of the Ostrovsky
equation, we know at least one null-mode is generated by the spatial translation invariance of the governing PDE \eqref{e:ostrov}.  Indeed, differentiating the profile
equation \eqref{e:profile} with respect to $\eta$ and integrating twice with respect to $\eta$ yields
\[
\widetilde{\mathcal{L}}[\phi]\phi'=0.
\]
Throughout this work, we make the non-degeneracy assumption that the above accounts for the entirety of the kernel of the Hessian operator.

\begin{assumption}[Nondegeneracy]\label{a:ker}
For a given $T$-periodic traveling wave solution $\phi$ of the Ostrovsky equation \eqref{e:ostrov}, the kernel of the Hessian operator
$\widetilde{\mathcal{L}}[\phi]$, acting on the space of mean-zero functions in $L^2_{\rm per}(0,T)$, satisfies
\[
\ker\left(\widetilde{\mathcal{L}}[\phi]\right)={\rm span}\left\{\phi'\right\}.
\]
\end{assumption}

\begin{remark}
We note that Assumption \ref{a:ker} has been analytically verified for both asymptotically small amplitude (i.e. Stokes) periodic traveling wave solutions
of the Ostrovsky equation \eqref{e:ostrov} as well as for wave trains in the weak rotation $\gamma\to 0^+$ limit: see \cite{BJP24,HSS17}.  Throughout our work,
we simply assume the above non-degeneracy condition holds.
\end{remark}

In our forthcoming analysis, an important role will be played by the non-symmetric operator
\begin{equation}\label{e:L_1}
\mathcal{L}[\phi]
:=\gamma+\partial_\eta^2\left(c-\phi-\beta\partial_\eta^2\right),
\end{equation}
acting on all of $L^2_{\rm per}(0,T)$.  The actions of $\mathcal{L}[\phi]$ and $\widetilde{\mathcal{L}}[\phi]$ are closely
related on the set of mean-zero functions.\footnote{Indeed, if $f\in L^2_{\rm per}(0,T)$ has mean-zero then $\mathcal{L}[\phi]f=-\partial_\eta^2\widetilde{\mathcal{L}}[\phi]f$.}
The following characterizes the kernel of $\mathcal{L}[\phi]$, as well as its adjoint, acting on all of $L^2_{\rm per} (0,T)$.

%

\begin{lemma}\label{L:non-deg}
Let $\phi$ be a $T$-periodic traveling wave solution of the Ostrovsky equation
\eqref{e:ostrov}.   Under Assumption \eqref{a:soln} and Assumption \eqref{a:ker}, we have
\[
\ker_{L^2_{\rm per}(0,T)}\left(\mathcal{L}[\phi]\right)={\rm span}\left\{\phi' \right\}
\]
and
\[
{\rm ker}_{L^2_{\rm per}(0,T)}\left(\mathcal{L}^\dag[\phi]\right)={\rm span}\left\{\partial_\eta^{-1}\phi\right\}
`\]
where here $\mathcal{L}[\phi]$ is defined as in \eqref{e:L_1}.
\end{lemma}

\begin{proof}
To characterize the kernel of $\mathcal{L}[\phi]$ we observe that  if $f\in L^2_{\rm per}(0,T)$ is such that $\mathcal{L}[\phi]f=0$, 
then we necessarily have $\int_0^Tf~d\eta=0$ and hence\footnote{Note that $f$ being mean-zero
implies that it is in the domain of $\widetilde{\mathcal{L}}[\phi]$.}
\[
0=\partial_\eta^{-2}\mathcal{L}[\phi]f=\widetilde{\mathcal{L}}[\phi]f.
\]
By Assumption \eqref{a:ker}, it follows that $f$ is proportional to $\phi'$, as desired.

Similarly, to characterize the kernel of the adjoint we suppose that $g\in L^2_{\rm per}(0,T)$ is such that $\mathcal{L}^\dag[\phi]g=0$, i.e. that
\[
\gamma g +\left(c-\phi-\beta \partial_\eta^2\right)g''=0.
\]
Differentiating the above twice yields
\[
0=\partial_\eta^2\mathcal{L}^\dag[\phi]g = -\mathcal{L}[\phi]g'',
\]
where the last equality follows by a direct calculation.
By our characterization of the kernel of $\mathcal{L}[\phi]$ above, it follows that $g''=\alpha\phi'$ for some $\alpha\in\RM$.  Noting that $\mathcal{L}^\dag[\phi]1=\gamma\neq 0$, it 
follows by integrating twice that $g=\alpha\partial_\eta^{-1}\phi$, as claimed.
\end{proof}

Finally, we note a technical result following from Lemma \ref{L:non-deg}, which will be used in several places in our forthcoming analysis.

\begin{corollary}
Under the hypotheses of Lemma \ref{L:non-deg} we have $\frac{\partial c}{\partial P}\neq 0$.
\end{corollary}

\begin{proof} 
Note that differentiating the profile equation \eqref{e:profile_stokes} with respect to $P$, keeping in mind that both the momentum $P$ and the frequency $k$ are independent variables
thanks to Assumption \ref{a:soln}, it follows that $\mathcal{L}[\phi]\phi_P=-k^2 c_P\phi''$.  Consequently, if $\frac{\partial c}{\partial P}= 0$ then
there would necessarily be a  second linearly independent element in the kernel of $\mathcal{L}[\phi]$, which is prohibited
under Assumption \ref{a:ker} by Lemma \ref{L:non-deg}.  It follows that under the hypotheses of Lemma \ref{L:non-deg} that we must have $\frac{\partial c}{\partial P}\neq 0$,
as claimed.
\end{proof}

\section{The Whitham Modulation Equations}
\label{s:Whitham}

In this section, we begin our study of the large-space/long-time dynamics of an arbitrary amplitude, slowly modulated periodic traveling wave solution of the Ostrovsky equation \eqref{e:ostrov}. A formal approach to study the dynamical behavior of such slowly modulated periodic traveling waves is to use multiple-scales / WKB methods to derive a 
dispersionless system of quasilinear PDEs, referred to as the Whitham modulation equations, 
that govern the slow evolution of the fundamental wave characteristics.\footnote{Here, as per our existence
theory above, the frequency $k$ and momentum $P$ of the underlying wave train.}   We note that while Whitham originally formulated this approach
in terms of averaged conservation laws, it was later shown to be equivalent to an asymptotic WKB reduction: see \cite{L66,W65,W74}.  Further, we find that the derivation through
multiple scales more directly aligns with our upcoming rigorous theory, which utilizes spectral perturbation theory to project the infinite dimensional (rigorous) spectral
 stability problem onto an appropriate finite-dimensional critical subspace.

We note that in \cite{WJ17} the authors derived and studied the Whitham modulation
equations for the Ostrovsky equation in a particular scaling of the weak-rotation limit $0<\gamma\ll 1$.  In this particular limit, the resulting averaged system is effectively a 
forced version of the Whitham modulation equations for the limiting KdV equation.  In the present work, we consider arbitrary $\gamma\neq 0$ which, as we will
see, changes the nature of the modulation system considerably.  For more details and discussion on this, see Remark \ref{r:WJ} below.

\subsection{Whitham Theory for the Ostrovsky Equation}\label{s:whitham_derivation}

To provide an asymptotic description of the slow modulation of periodic wave trains for the Ostrovsky equation \eqref{e:ostrov}, which occurs
at large space and time scales when compared to the local spatial and temporal scales of the wave, we begin by introducing  ``slow"  variables 
$(X,S)=(\eps x,\eps t)$ where here $0<|\eps|\ll 1$.  In the slow variables, the Ostrovsky equation \eqref{e:ostrov} becomes
\be
\label{e:ostrov_slow}
\eps^2u_{SX} + \eps^2(u_X)^2 + \eps^2uu_{XX} + \eps^4\beta u_{XXXX} = \gamma u.
\ee
Following Whitham's general method \cite{W74}, we seek 
a multiple-scales solution to \eqref{e:ostrov_slow} of the form 
\be
\label{e:W_expansion}
u(X,S;\eps) = u^0(X,S,\eps^{-1}\psi(X,S)) + \eps u^1(X,S,\eps^{-1}\psi(X,S)) + \mathcal{O}(\eps^2),
\ee 
where here the phase $\psi$ is chosen so that the $u^j$ are $1$-periodic functions of the third (fast) coordinate $\theta = \eps^{-1}\psi(X,S)$.
Substituting the expansion \eqref{e:W_expansion} into \eqref{e:ostrov_slow} and collecting like power of $\eps$ yields a hierarchy of equations in algebraic orders of $\eps$ that
must all be simultaneously satisfied.  Importantly, note that substituting \eqref{e:W_expansion} into \eqref{e:ostrov_slow} and integrating over a period in $\theta$ yields
\[
\int_0^1u^j(X,S,\theta)d\theta=0,~~j=0,1,2,\ldots
\]
and hence the infinite hierachy of equations must necessarily be solved in the mean-zero subspace of $L^2_{\rm per}(0,1)$.

At leading order, occurring at $\mathcal{O}(1)$, we obtain
\be
\label{e:O(1)}
\psi_S\psi_Xu^0_{\theta\theta} + \psi_X^2(u^0_{\theta})^2 + \psi_X^2u^0u^0_{\theta\theta} + \beta \psi_X^4u_{\theta\theta\theta\theta}^0 = \gamma u^0.
\ee 
Setting $k=\psi_X$ and $\omega=-\psi_S$ as the spatial and temporal frequencies of the modulation, respectively, and $c=\omega/k$ as the wave speed, the above is immediately recognized as the profile equation \eqref{e:profile}, rescaled for $1$-periodic functions. Note that, here, the $k$, $\omega$ and $c$ are now slowly evolving functions
of space and time.  Consequently, for a fixed $X$ and $S$ we may choose $u^0$ to be a periodic traveling wave solution of \eqref{e:ostrov}, hence of the form
\[
u^0(X,S,\theta)=\phi\left(\theta;k(X,S),P(X,S)\right)
\]
for some even solution $\phi$ of \eqref{e:profile}, again rescaled for $1$-periodic solutions.\footnote{As per our existence theory, we are assuming periodic solutions
of \eqref{e:profile} form a 2-parameter family parameterized by the frequency $k$ and the momentum $P$.  The wave speed $c$ and temporal frequency
$\omega$ above thus depend slowly through the evolution of $k$ and $P$.}  Further, the 
consistency condition $\psi_{XS}=\psi_{SX}$ implies the local wave number $k$ and the wave speed $c$ slowly evolve according
to the relation
\be
\label{e:consv_waves}
k_S + (kc)_X = 0,
\ee
which is often referred to as ``conservation of waves."

Continuing to study the above hierarchy of equations, at $\mathcal{O}(\eps)$ we find
\begin{multline*}
\psi_Xu^0_{S\theta}+\psi_{SX}u^0_{\theta}+\psi_Su^0_{\theta X}+\psi_S\psi_Xu^1_{\theta\theta}+2\psi_Xu^0_{\theta}u^0_X+2\psi_X^2u^0_{\theta}u^1_{\theta}+2\psi_Xu^0u^0_{X\theta}+\psi_{XX}u^0u^0_{\theta}
\\
+\psi_X^2u^0u^1_{\theta\theta}+\psi_X^2u^0_{\theta\theta}u^1+4\beta \psi_X^3u^0_{X\theta\theta\theta}+6\beta \psi_X^2\psi_{XX}u^0_{\theta\theta\theta}+\beta\psi_X^4u^1_{\theta\theta\theta\theta} = \gamma u^1.
\end{multline*}
Using our above identifications, the above can be rewritten as
\be
\label{e:O(eps)}
\mathcal{L}[u^0]u^1 = k\d_{\theta}u^0_S + F(u^0),
\ee
where here $\mathcal{L}[u^0]$ is the linear differential operator\footnote{Note that $\mathcal{L}[\phi]$ is the same operator defined in \eqref{e:L_1} but rescaled for $1$-periodic functions.} 
\begin{equation}\label{e:L0}
\mathcal{L}[u^0]:=\gamma+k^2\partial_\theta^2\left(c-u^0-\beta k^2\partial_\theta^2\right),
\end{equation}
equipped with $1$-periodic boundary conditions, and 
\[
F(u^0):= -(kcu^0_{\theta})_X+2ku^0_{\theta}u^0_X+2ku^0u_{X\theta}^0+k_Xu^0u^0_{\theta}+4\beta k^3u^0_{X\theta\theta\theta}+6\beta k^2k_Xu^0_{\theta\theta\theta}. 
\] 
contains all the nonlinear terms in $u^0$ and its derivatives.  
Viewing \eqref{e:O(eps)} as a forced linear equation in the vector space $L^2_{\rm per}(0,1)$ for the unknown $u^1$,
it follows by the Fredholm alternative that \eqref{e:O(eps)} is solvable 
in $L^2_{\rm per}(0,1)$ if and only if
\[
k\d_{\theta}u^0_S+F(u^0)\perp \ker_{L^2_{\rm per}(0,1)}\left(\mathcal{L}^\dag[u^0]\right).
\]  
From Lemma \ref{L:non-deg}, it follows that \eqref{e:O(eps)} is solvable provided that the single orthogonality condition
\be
\label{e:solvability} 
\left< \frac{1}{k}\d_{\theta}^{-1}u^0, k\d_{\theta}u^0_S+F(u^0) \right> = 0
\ee
is satisfied, where here and throughout $\left<\cdot,\cdot\right> := \left<\cdot,\cdot\right>_{L^2_{\rm per}(0,1)}$. 
Noting that
\[
\left<\frac{1}{k}\d_{\theta}^{-1}u^0,k\d_{\theta}u^0_S\right> = \left<u^0,-u_S^0\right> 
=-P_S,
\]
the above solvability condition is equivalent to the slow evolution equation
\begin{equation}\label{e:P_slow1}
P_S=\left< \frac{1}{k}\d_{\theta}^{-1}u^0, F(u^0) \right>
\end{equation}
for the momentum $P$. 

To write the right-hand-side of \eqref{e:P_slow1} in conservative form, first note that by integrating the rescaled 
profile equation \eqref{e:O(1)} and differentiating with respect to the slow-coordinate $X$ yields
\begin{align*}
-(kcu^0_{\theta})_X&= -k_Xu^0u^0_{\theta} - ku^0_Xu^0_{\theta} - ku^0u^0_{\theta X} - 3\beta k^2k_Xu^0_{\theta\theta\theta} - \beta k^3u^0_{\theta\theta\theta X}\\
&\qquad - \frac{\gamma}{k^2}k_X\d_{\theta}^{-1}u^0 + \frac{\gamma}{k}(\d_{\theta}^{-1}u^0)_X.
\end{align*}
Substituting this into the definition of $F(u^0)$ and simplifying gives  
\begin{multline*}
\left< \frac{1}{k}\d_{\theta}^{-1}u^0,F(u^0)\right> 
\\
= \left<\frac{1}{k}\d_{\theta}^{-1}u^0, ku^0_Xu^0_{\theta} + ku^0u^0_{\theta X} + 3\beta k^2k_Xu^0_{\theta\theta\theta} + 3\beta k^3u^0_{\theta\theta\theta X} - \frac{\gamma}{k^2}k_X\d_{\theta}^{-1}u^0 + \frac{\gamma}{k}(\d_{\theta}^{-1}u^0)_X\right>. 
\end{multline*} 
Using integration by parts, we find that
\[
\left<\frac{1}{k}\d_{\theta}^{-1}u^0, ku^0_Xu^0_{\theta}+ku^0u^0_{\theta X}\right> = \left<1, -(u^0)^2u^0_X\right> = \left<1, -\frac{1}{3}(u^0)^3\right>_X
\]
and, similarly,
\begin{align*}
\left<\frac{1}{k}\d_{\theta}^{-1}u^0, 3\beta k^2k_Xu^0_{\theta\theta\theta} + 3\beta k^3u^0_{\theta\theta\theta X}\right> &= \left<1, 3\beta kk_X(u^0_{\theta})^2 + 3\beta k^2u^0_{\theta}u^0_{\theta X}\right> \\
&= \left<1, \frac{3}{2}\beta k^2(u^0_{\theta})^2\right>_X.
\end{align*}
Finally, we also find that
\begin{align*}
\left<\frac{1}{k}\d_{\theta}^{-1}u^0, -\frac{\gamma}{k^2}k_X\d_{\theta}^{-1}u^0 + \frac{\gamma}{k}(\d_{\theta}^{-1}u^0)_X\right> &= \left<1, -\frac{\gamma}{k^3}k_X(\d_{\theta}^{-1})^2u^0 + \frac{\gamma}{k^2}\d_{\theta}^{-1}u^0(\d_{\theta}^{-1}u^0)_X\right> \\
&= \left<1, \frac{\gamma}{2k^2}(\d_{\theta}^{-1}u^0)^2\right>_X,
\end{align*}
and hence the slow evolution equation \eqref{e:P_slow1} can be rewritten in conservative form as
\begin{equation}\label{e:P_slow2}
P_S= \left<1,-\frac{1}{3}(u^0)^3+\frac{3}{2}\beta k^2(u_{\theta}^0)^2+\frac{\gamma}{2k^2}(\d_{\theta}^{-1}u^0)^2\right>_X.
\end{equation}

Taken together,  the consistency condition \eqref{e:consv_waves} and the solvability condition \eqref{e:solvability} gives the system of equations
\begin{equation}\label{e:Wsystem}
\left\{\begin{aligned}
k_S &= -(kc)_X,\\
P_S &= \left<1, -\frac{1}{3}(u^0)^3 + \frac{3}{2}\beta k^2(u^0_{\theta})^2 + \frac{\gamma}{2k^2}(\d_{\theta}^{-1}u^0)^2\right>_X,
\end{aligned}\right.
\end{equation}
which are the Whitham modulation equations for the Ostrovsky equation \eqref{e:ostrov}.  This system is expected to govern (at least to leading order) 
the slow evolution of the wave number $k$ and the momentum $P$ of a slow modulation of the underlying periodic traveling wave $u^0$.  
In particular, if $\phi_0(\cdot;k_0,P_0)$ is a $1/k$-periodic traveling wave solution of  \eqref{e:profile}, then the above multiple-scales analysis suggests that, at least
for short time, the Ostrovsky equation admits a slowly modulated periodic wave train solution of the form
\[
u(x,t;\eps)=\phi\left(\eps^{-1}\psi(\eps x,\eps t);k(\eps x,\eps t),P(\eps x,\eps t)\right)+\mathcal{O}(\eps)
\]
for $|\eps|\ll 1$, where the $k(\eps x,\eps t)$ and $P(\eps x,\eps t)$ are slowly varying functions about $k_0$ and $P_0$, respectively, that evolve
according to the Whitham modualtion system \eqref{e:Wsystem}.  As such, the stability of $\phi_0(\cdot;k_0,P_0)$ to slow modulations is expected to be
determined from the linearization of \eqref{e:Wsystem} about $\phi_0$.
Specifically, using the chain rule we rewrite the Whitham system \eqref{e:Wsystem} in the quasilinear form
\be
\label{e:quasilin}
\begin{pmatrix} k \\ P \end{pmatrix}_S = \mathcal{W}(u^0)\begin{pmatrix} k \\ P \end{pmatrix}_X,
\ee
where 
\be
\label{e:Wmat}
\mathcal{W}(u) = 
\left(\begin{array}{cc}-(kc)_k & -(kc)_P \\ \left<1,-\frac{1}{3}u^3+\frac{3}{2}\beta k^2u_{\theta}^2 + \frac{\gamma}{2k^2}(\d_{\theta}^{-1}u)^2\right>_k & \left<1,-\frac{1}{3}u^3+\frac{3}{2}\beta k^2u_{\theta}^2 + \frac{\gamma}{2k^2}(\d_{\theta}^{-1}u)^2\right>_P 
\end{array}\right).
\ee
it is natural to expect the stability of $\phi_0$ to slow modulations should be determined by the eigenvalues of the matrix $\mathcal{W}(\phi_0)$.  
Indeed, noting that $\phi_0$ is an equilibrium solution\footnote{Note that $\phi_0$ is a constant with respect to the slow variables $(X,S)$.} of \eqref{e:Wsystem},
linearizing \eqref{e:Wsystem} about $\phi_0$ yields a constant-coefficient linear evolution equation for $(k,P)$ whose eigenvalues will be of the form
\[
\nu_j(\xi)=i\xi\alpha_j,~~j=1,2
\]
where the $\{\alpha_j\}_{j=1}^2$ are precisely the eigenvalues of the matrix $\mathcal{W}(\phi_0)$ and $\xi\in\RM$ is a parameter.  It follows
that if the Whitham system is weakly hyperbolic at $\phi^0$, meaning the eigenvalues of $\mathcal{W}(\phi_0)$ are both real, then the eigenvalues of the linearization 
of \eqref{e:quasilin} are purely imaginary, indicating
a marginal (spectral) stability.   Conversely, if the Whitham system is elliptic at $\phi_0$, meaning the eigenvalues of $\mathcal{W}(\phi_0)$ have non-zero imaginary parts, 
then the linearization of \eqref{e:quasilin} about $\phi_0$ has eigenvalues with positive real part, indicating (spectral) instability of $\phi_0$.

\

The main goal of this work is to rigorously validate the above predictions.  Specifically, we will demonstrate that a necessary condition for a given periodic
traveling wave $\phi_0(\cdot;k_0,P_0)$ to be spectrally stable is the weak-hyperbolicity of the Whitham modulation system \eqref{e:Wsystem} at $(k_0,P_0)$,
while ellipticity of \eqref{e:Wsystem} at $(k_0,P_0)$ is sufficient to conclude spectral instability of $\phi_0$.

Before continuing, we make two important remarks.

\begin{remark}\label{r:re-param}
It is important to note that Whitham's theory of modulations naturally describes the slow modulation of a periodic traveling wave solution
in terms of the slow evolution of the spatial frequency and conserved quantities of the background wave (in this case, the momentum).
This makes it natural to describe the manifold of periodic traveling wave solutions to the Ostrovsky equation in terms of $k$ and $P$,
which is effectively the content of Assumption \ref{a:soln}.  Of course, it is natural to consider what would happen if a different parameterization
for the manifold of solutions is used.

Observe that if one assumes the periodic solutions of \eqref{e:profile} are parameterized by two parameters $(a_1,a_2)\in\RM$ via  $\phi=\phi(\cdot;a_1,a_2)$,
with period $T=T(a_1,a_2)$ and wave speed $c=c(a_1,a_2)$, then the Whitham system \eqref{e:quasilin} can be rewritten 
as a slow evolution equation for the parameters $a_j$ via
\[
\left(\begin{array}{c}a_1\\a_2\end{array}\right)_S
=\left(\frac{\partial(k,P)}{\partial(a_1,a_2)}\right)^{-1}\mathcal{W}(u^0)\left(\frac{\partial(k,P)}{\partial(a_1,a_2)}\right)
	\left(\begin{array}{c}a_1\\a_2\end{array}\right)_X.
\]
In particular, the above only defines a slow evolution equation for $(a_1,a_2)$ provided that the Jacobian $\frac{\partial(k,P)}{\partial(a_1,a_2)}$
is non-singular, i.e. if and only if the waves $\phi(\cdot;a_1,a_2)$ can, in fact, be smoothly reparameterized in terms
of the frequency $k$ and the momentum $P$.  In this way, we see that Assumption \ref{a:soln} is effectively \emph{required} in order
for the Whitham modulation system to describe a slow evolutionary equation on the manifold of periodic traveling wave solutions to the Ostrovsky equation.
See also Remark \ref{r:exist_assumption} above.  Further, the above calculation shows that the hyperbolicity or ellipticity of the resulting Whitham
system is independent of how one parameterizes the manifold of periodic traveling wave solutions.
\end{remark}

\begin{remark}\label{r:WJ}
It should be emphasized that our application of Whitham's theory of modulations to the Ostrovsky equation
\eqref{e:ostrov} is conducted for general, fixed values of $\gamma$.  
In \cite{WJ17}, Whitfield \& Johnson present a derivation of the Whitham modulation system for the Ostrovsky equation
in a particular weak rotation limit $0<\gamma\ll 1$.  In particular, using language and notation from our work above, the authors
in \cite{WJ17} make the additional assumption in \eqref{e:ostrov_slow} that $\gamma=\mathcal{O}(\eps^2)$.  The effect
of this additional restriction is substantial: most notably, the modulation system derived in \cite{WJ17} is three-dimensional,
describing the evolution of three slowly modulating parameters defining the underlying periodic traveling wave solutions.
Below, we outline the main differences between our derivation and that in \cite{WJ17}.

To compare our results, note that taking $\gamma=\mathcal{O}(\eps^2)$ in the slow evolution equation \eqref{e:ostrov_slow} above
implies that the $\mathcal{O}(1)$ equation in \eqref{e:O(1)} becomes 
\[
\psi_S\psi_Xu^0_{\theta\theta} + \psi_X^2(u^0_{\theta})^2 + \psi_X^2u^0u^0_{\theta\theta} + \beta \psi_X^4u_{\theta\theta\theta\theta}^0 = 0
\]
which, under the standard identifications $k=\psi_X$, $\omega=-\psi_S$ and $c=\omega/k$ as above, is directly integrated
to yield the traveling wave profile equation for the classical (integrable) KdV equation.  The KdV equation is known
to admit a three parameter family of periodic traveling wave solutions smoothly parameterized by the frequency $k$, the
momentum $P$, and the mass $M=\int_0^1 u^0(\theta)d\theta$.   Continuing in the asymptotic expansion,
the forced linear equation \eqref{e:O(eps)} for the unknown $u^1$ remains structurally the same, except that the linear operator $\mathcal{L}[u^0]$
is replaced by the operator
\[
\mathcal{A}[u^0]:=k^2\partial_\theta^2\left(c-u^0-\beta k^2\partial_\theta^2\right),
\]
and the forcing term $F(u^0)$ is modified appropriately.  Importantly, the operator $\mathcal{A}[u^0]$ is recognized
precisely as the derivative of the lineraized KdV operator, and hence its adjoint is known to have a two-dimensional
kernel spanned by the constant function $1$ and the function $\partial_\theta^{-1}u^0$.  The presence of the additional
adjoint eigenfunction, namely the (non mean-zero) constant function $1$, directly accounts for the presence of the additional dimension
in the Whitham modulation system derived in \cite[Section 2]{WJ17} compared to our work above.

As stressed above, our derivation in Section \ref{s:whitham_derivation} holds for general fixed $\gamma\neq 0$, while the derivation
in \cite{WJ17} is conducted in the asymptotic (weak-rotation) limit $\gamma=\mathcal{O}(\eps^2)$.  
Most importantly, our work shows that for a \emph{fixed} value of $\gamma\neq 0$ that the 
associated Whitham modulation system is necessarily only two-dimensional.
\end{remark}

\subsection{Stokes Wave Analysis: Whitham Theory \& Lighthill's Criteria}\label{s:stokes}

Evaluating the eigenvalues of the linearized Whitham matrix \eqref{e:Wmat} is incredibly difficult, if not impossible, to do analytically for general
periodic traveling wave solutions of the Ostrovsky equation.  
However, the calculations are relatively straight forward in the 
special case where $\phi$ is a Stokes wave, i.e. when $\phi$ is a periodic
traveling wave of \eqref{e:ostrov} with asymptotically small amplitude of oscillation.  Here we 
study the hyperbolicity/ellipticity of the Whitham system \eqref{e:Wsystem} when applied to the Stokes wave solutions
for the Ostrovsky equation \eqref{e:ostrov} and connect to other previous works regarding the spectral stability of such waves.
In particular, in this section we demonstrate that the Whitham system derived above will indeed be (strictly) hyperbolic at some waves while it 
will be elliptic at others.  See Remark \ref{r:stokes} as well.

As discussed in Section \ref{s:existence}, using a standard Lyapunov-Schmidt reduction it was proven in \cite{BJP24} that 
for each fixed $k>0$ there exists a one-parameter family of smooth solutions of \eqref{e:profile} of the form
$u(x,t)=\phi\left(k(x-c(a,k)t);a,k\right)$ for $a\in\RM$ with $|a|\ll 1$, where the $\phi(\cdot;a,k)$ are $1$-periodic,
even, and smooth in its argument, and $c(a,k)$ is even in $a$.  Further, both $\phi(\cdot;a,k)$ and $c(a,k)$
depend analytically on $a$ and $k$, and admit asymptotic expansions as in \eqref{e:stokes_waves}.
Substituting the  expansions \eqref{e:stokes_waves} into the Whitham modulation system \eqref{e:Wsystem} yields 
\begin{equation}\label{e:W_Stokes}
\left\{\begin{aligned}
	&k_S+\left(\omega_0+a^2 \omega_2\right)_X=\mathcal{O}(a^3)\\
	&\left(\frac{1}{4}a^2\right)_S=\left(\frac{48\beta k^4\pi^4+\gamma}{16k^2\pi^2}a^2\right)_X + \mathcal{O}(a^3)
	\end{aligned}\right.
\end{equation}
where here the temporal frequency\footnote{Note this identification is consistent with those made in our application of Whitham theory.  See the definitions
of the slowly varying temporal frequency $\omega$ and wave speed $c$ directly below \eqref{e:O(1)}.} $\omega$ and its expansion are given by
\begin{equation}\label{e:tfreq}
\omega(a,k)=kc(a,k)=\omega_0(k)+a^2\omega_2(k)+\mathcal{O}(a^4).
\end{equation}
Using the chain rule, the system \eqref{e:W_Stokes} can be rewritten as the quasilinear system
\[
\left(\begin{array}{c}k\\a\end{array}\right)_S = B(k,a)\left(\begin{array}{c}k\\a\end{array}\right)_X
\]
where here
\[
B(k,a)=\left(\begin{array}{cc}
					-\frac{\partial\omega_0}{\partial k} & 0\\
					0 & -\frac{\partial\omega_0}{\partial k}
				\end{array}\right)
				+
				a\left(\begin{array}{cc}
					0 & -2\omega_2\\
					 -\frac{1}{2}\frac{\partial^2\omega_0}{\partial k^2} & 0
				\end{array}\right)
				+\mathcal{O}(a^2).
\]
A straightforward computation shows that the eigenvalues of the linearized Whitham matrix $B(k,a)$ are given explicitly by
\[
\lambda_{\pm}(k,a)=-\frac{\partial\omega_0}{\partial k}\pm a\sqrt{\omega_2\frac{\partial^2\omega_0}{\partial k^2} }+\mathcal{O}(a^2).
\]
and hence the Whitham system about a $1/k$-periodic Stokes wave is hyperbolic (suggestive of modulational stability) when $\omega_2(k)\omega_0''(k)>0$,
while it is elliptic (suggestive of modulational instability) when $\omega_2(k)\omega_0''(k)<0$.

\begin{remark}\label{r:stokes2}
Note that, explicitly, we have
\[
\omega_0''(k)\omega_2(k)=\frac{1}{3}\left(\frac{\gamma-48\beta\pi^4k^4}{\gamma+64\beta\pi^4k^4}\right).
\]
For both $\beta>0$ and $\beta<0$, there exists a crticial frequency cut-off $k_c=k_c(\gamma,\beta)$ such that
the Whitham system \eqref{e:Wsystem} is hyperbolic for $0<k<k_c$ and elliptic for $k>k_c$.  We note that for both signs of $\beta$,
the critical frequency scales like
\[
k_c\sim\left(\frac{\gamma}{|\beta|}\right)^{1/4}.
\]
In particular,  we note that\footnote{Said differently, formally taking $\beta=0$ gives $\omega_0''(k)\omega_2(k)=1/3$, indicating stability for all frequencies $k$.} 
$k_c\to\infty$ as $\beta\to 0$, consistent with the work \cite{JP16} which shows that all Stokes waves
are nonlinearly orbitally stable to a wide class of periodic perturbations.\footnote{In particular, they are shown to be stable to all subharmonic perturbations.}
On the other hand, $k_c\to 0^+$ as $\gamma\to 0^+$, which of course is not consistent with the known modulational stability of Stokes waves (and, in fact, all periodic
traveling waves) in the limiting KdV equation.\footnote{Said differently, formally taking $\gamma=0$ yields $\omega_0''(k)\omega_2(k)=-\frac{1}{4}$, indicating instability.  Of course, it is
well known that all periodic traveling waves of the limiting KdV equation are stable to a wide class of perturbations: see \cite{J09,ND10} and references therein.}
\end{remark}

As discussed in detail in \cite{BJP24}, the above asymptotics recover the well-known Lighthill
criteria for modulational instability of small-amplitude, periodic wave trains of the Ostrovsky equation.  
Indeed, using formal asymptotc methods, it was shown in \cite{GH08,GSA16} that if $u$ is a small amplitude, weakly nonlinear periodic solution of \eqref{e:ostrov} of the form\footnote{Here, c.c.
denotes the complex conjugate of the previous term.}
\[
u(x,t)=\left(Ae^{2\pi i\theta}+c.c.\right) + \left(A_2 e^{4\pi i\theta}+c.c.\right)+\ldots,
\]
where $\theta=kx-\widetilde{\omega}(|A|,k)t$ and $\widetilde{\omega}(|A|,k)$ denotes the nonlinear dispersion relation with $|A|=\eps\ll 1$ is a slowly varying function
 of space and time with $A_2=\mathcal{O}(\eps^2)$, then in the moving frame $z=x-c_g(k)t$ the leading order $A$ term satisfies the nonlinear Schr\"odinger equation
 \begin{equation}\label{e:NLS}
 iA_t+\frac{1}{2}\omega_0''(k)A_{zz}-\omega_2(k)|A|^2A=0,
 \end{equation}
where here $\omega_0(k)$ and $\omega_2(k)$ are as in \eqref{e:tfreq} and are related to $\widetilde{\omega}(|A|,k)$ through the expansion
\[
\widetilde{\omega}(|A|,k)=\omega(k)+\eps^2\omega_2(k)+\mathcal{O}(\eps^4),~~\eps\to 0.
\]
The Lighthill criteria states that such a weakly nonlinear wave train is modulationally unstable provided that the NLS equation \eqref{e:NLS} is focusing, i.e. provided
that
\[
\omega_0''(k)\omega_2(k)<0,
\]
while it is modulationally stable if \eqref{e:NLS} is defocusing, i.e. if $\omega_0''(k)\omega_2(k)>0$.

\

In summary, the Whitham theory for Stokes waves above agrees precisely with the Lighthill criteria coming from a separate, albeit also formal,
weakly nonlinear analysis.  In the recent work \cite{BJP24}, the authors studied the rigorous (spectral) modulational instability of 
Stokes waves for the Ostrovsky equation through the use of rigorous
spectral perturbation theory.  In that work, the formal Lighthill criteria was rigorously connected to the spectral stability of the underlying wave train,
showing that such waves are spectrally stable in a neighborhood of the origin in the spectral plane if $\omega_0''(k)\omega_2(k)>0$ while they
exhibit a modulational (i.e. side-band) spectral instability when $\omega_0''(k)\omega_2(k)<0$.
As we have just connected Whitham's theory of modulations in this case also to the Lighthill criteria, it follows that \cite{BJP24} (and more generally
our main result Theorem \ref{t:main}) provides a rigorous connection between Whitham's theory of modulations to the 
spectral stability of Stokes waves for the Ostrovsky equation.

Recalling from the introduction, the main goal of this work is to extend the above observation by rigorously connecting the stability predictions from Whitham's theory of modulations
to the rigorous spectral stability problem for \emph{all} periodic wave train solutions of the Ostrovsky equation \eqref{e:ostrov}, i.e. for periodic
traveling waves outside of the Stokes wave asymptotic regime.

\section{Rigorous Modulational Stability Theory}
\label{s:modtime}

In this section, we conduct a mathematically rigorous investigation into the spectral stability problem for periodic traveling wave solutions
of the Ostrovsky equation \eqref{e:ostrov}.  Throughout, we will focus on the stability to a particular subclass
of \emph{localized} (i.e. $L^2(\RM)$) perturbations on the line.  As a first step, we will discuss the required functional analytic
setup regarding linear operators with periodic coefficients and their spectrum when acting on $L^2(\RM)$.  We will, specifically, use Floquet-Bloch
theory to provide a full description of their $L^2(\RM)$-stability spectrum in terms of the periodic eigenvalues of a one-parameter family of 
(Bloch) linear operator pencils.
We will then begin our analysis by characterizing the co-periodic generalized kernel of the appropriate operator pencil.
Using functional analytic spectral perturbation theory in the context of linear operator pencils, we then 
rigorously obtain an asymptotic description of the $L^2(\RM)$-stability spectrum in a small neighborhood of the origin in the spectral plane.

\subsection{Linearization \& Spectral Decomposition}

Let $\phi(\cdot;k,P)$ be a $T=1/k$-periodic traveling wave solution of \eqref{e:ostrov} with momentum $P$.  
Throughout our work, we 
operate under Assumption \ref{a:soln}, and hence we will assume that (locally) the set of periodic traveling wave solutions
near $\phi$ are smoothly parameterized\footnote{Up to spatial translations.} by the frequency $k$ and the momentum $P$.
To emphasize the dependence on the frequency $k$, we rescale the spatial variable as $y=kx$, so that $T=1/k$ periodic 
traveling wave solutions of \eqref{e:ostrov} correspond to $1$-periodic traveling wave solutions of the rescaled evolution equation 
\be
\label{e:ostrov_scale}
k(u_t + kuu_y + \beta k^3u_{yyy})_y = \gamma u.
\ee
Introducing the traveling coordinate $\theta = k(x-ct)=y-\omega t$, where $\omega = kc$ is the temporal frequency, it follows 
that traveling wave solutions of \eqref{e:ostrov_scale} correspond to solutions of the form
\[
u(y,t) = \phi(y-\omega t),
\]
and hence the given periodic traveling wave $\phi(\theta)$ is necessarily a $1$-periodic equilibrium solution of the rescaled evolutionary equation
\be
\label{e:ostrov_station}
k(u_t - kcu_{\theta} + kuu_{\theta} + \beta k^3u_{\theta\theta\theta})_{\theta} = \gamma u.
\ee
In particular, $\phi$ is a $1$-periodic solution to the rescaled profile equation
\be
\label{e:profile_scale}
k^2(-c\phi' + \phi\phi' + \beta k^2\phi''')' = \gamma \phi,
\ee
where here and forward we use $'$ to denote differentiation with respect to $\theta$ and the wave speed is now considered to be a (smooth) function of $k$ and $P$.

To study the stability of $\phi$ to localized perturbations, we linearize \eqref{e:ostrov_station} about $\phi$ to arrive at the linear evolution equation
\[
k(v_t - kcv' + k(\phi v)' + \beta k^3v''')' = \gamma v
\]
governing the perturbation $v(\theta,t)$.  Seeking solutions of the form $v(\theta,t)=e^{\lambda t}\tilde{v}(\theta)$ with $\lambda\in\CM$ and $\tilde{v}\in L^2(\RM)$
leads to the spectral problem 
\begin{equation}\label{e:gen_ev1}
\mathcal{L}[\phi]\tilde{v}=\lambda k\partial_\theta\tilde{v}
\end{equation}
acting on $L^2(\RM)$, where here 
\[
\mathcal{L}[\phi]=\gamma+k^2\partial_\theta^2\left(c-\phi-\beta k^2\partial_\theta^2\right)
\]
is considered as a closed, densely defined linear operator on $L^2(\RM)$.  In particular, we note that the operator $\mathcal{L}[\phi]$ defined
above is the same as that already encountered in our derivation of the Whitham system (see Section \ref{s:Whitham}) and whose kernel
properties were studied in Lemma \ref{L:non-deg}.

As is standard, the periodic traveling wave $\phi$ is said 
to be spectrally stable with respect to square integrable perturbations if 
\[
\sigma_{L^2(\RM)}\left(\mathcal{L}[\phi];k\partial_\theta\right)\subset\left\{\lambda\in\CM:\Re(\lambda)\leq 0\right\}.
\]
Otherwise, the wave $\phi$ is said to be spectrally unstable to square integrable perturbations.  Note that since 
\eqref{e:gen_ev1} is invariant under the transformation $(\tilde{v},\lambda)\mapsto (\overline{\tilde{v}},\overline{\lambda})$, as well
as the transformation $(\theta,\lambda)\mapsto (-\theta,-\lambda)$, the spectrum associated with 
\eqref{e:gen_ev1} 
is necessarily
symmetric with respect to reflections about both the real and imaginary axes.  Consequently, $\phi$ is spectrally stable to square integrable
perturbations if and only if
\[
\sigma_{L^2(\RM)}\left(\mathcal{L}[\phi];k\partial_\theta\right)\subset i\RM.
\]

\begin{remark}
We note that since $\partial_\theta$ fails to be invertible on $L^2(\RM)$, due to the presence of essential spectrum along the entire imaginary axis,
the spectral problem \eqref{e:gen_ev1} can not be recast as a standard spectral problem for the 
operator\footnote{We hesitate to even write this here, since the operator $\partial_\theta^{-1}$ is not even well-defined in 
this context.} ``$\partial_\theta^{-1}\mathcal{L}[\phi]$."  As such, in the forthcoming work we will consider directly the generalized eigenvalue
problem \eqref{e:gen_ev1}.
\end{remark}

We now turn to studying the spectrum associated with \eqref{e:gen_ev1}.  A detailed discussion regarding the structure of the spectrum
can be found in Appendix \ref{s:pencil}.  Here, we only summarize the main ideas.  
Since the operator $\mathcal{L}[\phi]$ has $1$-periodic coefficients, 
Floquet theory implies that solutions of \eqref{e:gen_ev1} can not be integrable on $\RM$
and, at best, they can be bounded functions on the line: see, for example, \cite{BHJ16,JP20,RS78}.  It follows
that the $L^2(\RM)$-spectrum of \eqref{e:gen_ev1} consists entirely of essential spectrum.
Additionally, any bounded solution (up to linear combinations) of \eqref{e:gen_ev1} is necessarily of the form
\[
\tilde{v}(\theta)=e^{i\xi\theta}w(\theta)
\]
for some $w\in L^2_{\rm per}(0,1)$ and $\xi\in[-\pi,\pi)$.  
Substituting the above decomposition into \eqref{e:gen_ev1} leads to the consideration of the one-parameter family of spectral problems
\begin{equation}\label{e:gen_ev_bloch}
\lambda k\left(\partial_\theta+i\xi\right)w=e^{-i\xi \theta}\mathcal{L}[\phi]e^{i\xi \theta}w=: \mathcal{L}_\xi[\phi]w~~\xi\in[-\pi,\pi)
\end{equation}
considered on $L^2_{\rm per}(0,1)$.  The parameter $\xi$ is referred to as the Bloch frequency, or sometimes the Bloch parameter, 
and  $\mathcal{L}_\xi[\phi]$ denotes the one-parameter 
family of so-called Bloch operators acting on $L^2_{\rm per}(0,1)$.  
Characterizing the spectrum of \eqref{e:gen_ev_bloch} in terms of a periodic Evans function (see Lemma \ref{L:evans}), 
we find for each $\xi\in[-\pi,\pi)$ that the $L^2_{\rm per}(0,1)$-spectrum of the linear operator pencil \eqref{e:gen_ev_bloch} is 
comprised entirely of isolated eigenvalues of finite multiplicity.  
Putting everything together, it follows that $\lambda\in\CM$ belongs to the $L^2(\RM)$-spectrum
for the linear operator pencil \eqref{e:gen_ev1} if and only if there exists a $\xi\in[-\pi,\pi)$ and a non-trivial solution $w\in L^2_{\rm per}(0,1)$ to 
the equation \eqref{e:gen_ev_bloch}.  In particular, we have the following result.

\begin{proposition}\label{P:Floquet}
For the spectral problem \eqref{e:gen_ev1}, the following results hold:
\begin{itemize}
\item[(1)]  The $L^2(\RM)$-spectrum associated to \eqref{e:gen_ev1} is entirely essential, i.e. the $L^2(\RM)$-point spectrum is empty.
\item[(2)] We have $\lambda\in \sigma_{L^2(\RM)}\left(\mathcal{L}[\phi];k\partial_\theta\right)$ 
if and only if there exists a $\xi\in[-\pi,\pi)$ such that $\lambda\in\sigma_{L^2_{\rm per}(0,1)}\left(\mathcal{L}[\phi];k(\partial_\theta+i\xi)\right)$.
\item[(3)]  For each $\xi\in[-\pi,\pi)$, the set  $\sigma_{L^2_{\rm per}(0,1)}\left(\mathcal{L}[\phi];k(\partial_\theta+i\xi)\right)$ consists
entirely of isolated eigenvalues of finite (algebraic) multiplicity.
\item[(4)] We have the  spectral decomposition
\[
\sigma_{L^2(\RM)}\left(\mathcal{L}[\phi];k\partial_\theta\right)=\bigcup_{\xi\in[-\pi,\pi)}\sigma_{L^2_{\rm per}(0,1)}\left(\mathcal{L}_\xi[\phi];k(\partial_\theta+i\xi)\right).
\]
\end{itemize}
\end{proposition}

Again, readers should consult Appendix \ref{s:perttheory} for details.
By our above definitions, it follows that a periodic traveling wave $\phi$ is spectrally stable with respect to square integrable perturbations provided
that for each $\xi\in[-\pi,\pi)$ the $1$-periodic Bloch (linear operator pencil) eigenvalue problem \eqref{e:gen_ev_bloch} has all of its eigenvalues
along the imaginary axis.  
Verifying this condition for all $\xi$ is clearly a daunting, if not impossible, task in general.  For the purposes of modulational stability analysis and for developing
a connection to Whitham's theory of modulations, however, it is well-known that one need only consider  
 \eqref{e:gen_ev_bloch} for $|\xi|\ll 1$ and only for $\lambda$ in a sufficiently small neighborhood of the origin in the spectral plane.  
As motivation for this note that, from the above identifications, a non-trivial $1$-periodic solution of \eqref{e:gen_ev_bloch}
induces a non-trivial solution to the problem
\[
\left\{\begin{aligned}
&\mathcal{L}[\phi]\tilde{v}= \lambda k\partial_\theta\tilde{v},\\
&\tilde{v}(\theta+1)=e^{i\xi}\tilde{v}(\theta).
\end{aligned}\right.
\]
As such, we see that taking $\xi=0$ corresponds to the stability of $\phi$ to perturbations with the same periodic structure as $\phi$ itself, i.e. to co-periodic perturbations.  
Similarly, $0<|\xi|\ll 1$ corresponds to long wave perturbations of $\phi$, which are observable on large space and time scales (relative to the internal frequency of oscillation
of $\phi$ itself).  Modulations of $\phi$ represents a special subclass of long-wave perturbations where the effect of the perturbation is to allow
the fundamental wave characteristics (here, owing to our existence theory, the frequency $k$, momentum $P$, and the spatial translation invariance) to slowly vary in time, 
i.e. to modulate.  As we will see, variations of the fundamental wave characteristics of $\phi$ are encoded in the $1$-periodic generalized kernel of 
associated to \eqref{e:gen_ev_bloch} with $\xi=0$.
From these considerations, it is natural to expect that the stability of $\phi$ to slow modulations is determined by studying 
the spectrum associated to \eqref{e:gen_ev_bloch} for 
$|(\lambda,\xi)|\ll 1$.

With the above in mind, we now provide a precise definition of our (spectral) notion of modulational stability.

\begin{definition}[Modulational Stability/Instability]\label{d:MI}
A given periodic traveling wave solution $\phi(\cdot;k,P)$ of \eqref{e:ostrov} is (spectrally) modulationally stable if there exists an open neighborhood $\mathcal{N}\subset\CM$
of the origin $\lambda=0$ and a $\xi_0>0$ such that 
\[
\sigma_{L^2_{\rm per}(0,1)}\left(\mathcal{L}_\xi[\phi];k(\partial_\theta+i\xi)\right)\cap \mathcal{N}\subset i \RM 
\]
for all $|\xi|<\xi_0$.
 Otherwise, the wave $\phi$ is said to be (spectrally) modulationally unstable.
\end{definition}

\begin{remark}\label{r:spec_sym}
It is straightforward to see that the spectral problem \eqref{e:gen_ev_bloch} is invariant under the transformation
\begin{equation}\label{e:spec_sym}
\left(w,\lambda,\xi,\theta\right)\mapsto\left(\overline{w},-\overline{\lambda},\xi,-\theta\right)
\end{equation}
and hence for a fixed $\xi$ the eigenvalues associated with \eqref{e:gen_ev_bloch} are symmetric with respect to reflections
about the imaginary axis.  Similarly, \eqref{e:gen_ev_bloch} is invariant under the transformation
\[
\left(w,\lambda,\xi\right)\mapsto\left(\overline{w},\overline{\lambda},-\xi\right)
\]
and hence
\[
\sigma_{L^2_{\rm per}(0,1)}\left(\mathcal{L}_\xi[\phi];k(\partial_\theta+i\xi)\right) =\overline{\sigma_{L^2_{\rm per}(0,1)}\left(\mathcal{L}_{-\xi}[\phi];k(\partial_\theta-i\xi)\right)}.
\]
Consequently, when performing a modulational stability analysis it is sufficient to consider $0<\xi\ll 1$ only.  While this observation is useful,
especially in numerical computations, our analysis here does not explicitly rely on it.
\end{remark}

In what follows, we will conduct a mathematically rigorous study of the eigenvalues of the (Bloch) spectral problem
\eqref{e:gen_ev_bloch} for $|(\lambda,\xi)|\ll 1$.  
As a first step, we will study the generalized kernel in $L^2_{\rm per}(0,1)$ for the unmodulated (i.e. $\xi=0$) problem.  As a consequence
of Assumption \ref{a:soln} regarding the smooth parametrization of the manifold of smooth periodic traveling wave solutions of the Ostrovsky equation,
we will show that, generically, $0\in\sigma_{L^2_{\rm per}(0,1)}\left(\mathcal{L}_0[\phi];k\partial_\theta\right)$ with algebraic multiplicity
two and geometric multiplicity one.
We will then use spectral perturbation theory to track the two $1$-periodic eigenvalues $\{\lambda_j(\xi)\}_{j=1}^2$  near the origin 
as they bifurcate from the $(\lambda,\xi)=(0,0)$ state.  
To this end, we will construct asymptotic approximations of the left and right projection operators onto the total two-dimensional eigenspace. 
This work will culminate in the identification of the leading order asymptotics of the eigenvalues $\lambda_j(\xi)$ in terms of the 
eigenvalues of an explicit $2\times 2$ matrix.

\begin{remark}\label{r:no_invert}
One may note that for each fixed, small $\xi\neq 0$ that the operator $(\partial_\theta+i\xi)$ is actually invertible on $L^2_{\rm per}(0,1)$.  Consequently,
for  such small, non-zero Bloch frequencies the (operator pencil) eigenvalue problem \eqref{e:gen_ev_bloch} can be rewritten
as the standard eigenvalue problem
\[
k^{-1}(\partial_\theta+i\xi)^{-1}\mathcal{L}_\xi[\phi]w=\lambda  w,~~\xi\neq 0.
\]
While this may seem convenient at first, in our forthcoming analysis we require analytic expansions of the various linearized operators involved about $\xi=0$.
It is quickly observed that $\mathcal{L}_\xi[\phi]$ does not map into the mean-zero subspace of $L^2_{\rm per}(0,1)$ for $0<|\xi|\ll 1$,
and hence the operator $\partial_\theta^{-1}$ (or any positive powers of it) is not in general well-defined on the range of $\mathcal{L}_\xi[\phi]$.  
As a result, the operator $(\partial_\theta+i\xi)^{-1}$ can not be analytically expanded about $\xi=0$.
Consequently, moving forward we will continue to work directly with the (operator pencil) eigenvalue problem \eqref{e:gen_ev_bloch}.
As we will see in the next section, however, the projection of the operator $(\partial_\theta+i\xi)$ onto the (modulational) subspace
bifurcating from $(\lambda,\xi)=(0,0)$ state is in fact invertible with its inverse being analytic in $\xi$ for $|\xi|\ll 1$.  This is a key feature of our forthcoming spectral
perturbation analysis.
\end{remark}

\subsection{Analysis of the Unmodulated Problem}

As discussed above, our first step is to study the $1$-periodic generalized kernel associated
with the linear operator pencil associated with \eqref{e:gen_ev_bloch}, as well as its adjoint problem, at $\xi=0$.
This will serve as the base point for our forthcoming spectral perturbation analysis.

To this end let $\phi=\phi(\cdot;k,P)$ be a $1$-periodic solution of the profile equation \eqref{e:profile_scale}.
Recalling that $c=c(k,P)$, differentiating the profile equation \eqref{e:profile_scale} gives the identities
\be
\label{e:Lident}
\mathcal{L}_0[\phi]\phi' = 0, \qquad \mathcal{L}_0[\phi]\phi_P = -k^2c_P\phi'',
\ee
where now $'$ denotes differentiation with respect to $\theta$.  Noting that the frequency $k$ is an independent variable, 
it follows that the functions $\phi'$ and $\phi_P$ are both smooth $1$-periodic functions 
and hence \eqref{e:Lident} implies that $\lambda=0$ is a $1$-periodic eigenvalue of the (operator pencil) eigenvalue
problem \eqref{e:gen_ev_bloch} at $\xi=0$ with algebraic multiplicity at least two, and that $\phi_P$ is a generalized
eigenfunction.

Additionally, note that if $f\in L^2_{\rm per}(0,1)$ has mean-zero then we have the identity
\begin{equation}\label{e:L_adj_identity}
\begin{aligned}
\mathcal{L}_0^\dag[\phi]f&=\gamma f+k^2\left(c-\phi-\beta k^2\partial_\theta^2\right)f''\\
	&=\partial_\theta^{-2}\left(\gamma f''+k^2\partial_\theta^2\left(c-\phi-\beta k^2\partial_\theta^2\right)f''\right)\\
	&=\partial_\theta^{-2}\mathcal{L}[\phi]\partial_\theta^2 f.
\end{aligned}
\end{equation}
Since both $\phi'$ and $\phi_P$ are necessarily mean-zero functions,\footnote{Differentiating the relation 
$\int_0^1\phi(\cdot;k,P)d\theta=0$ with respect to $P$, and keeping in mind
that the frequency $k$ and momentum $P$ are independent variables, it follows that $\int_0^1\phi_P~d\theta=0$.} the identity \eqref{e:L_adj_identity} immediately implies
that
\be
\label{e:Ldagident}
\mathcal{L}_0^\dag[\phi]\partial_\theta^{-1}\phi=0,\qquad \mathcal{L}_0^\dag[\phi]\partial_\theta^{-2}\phi_P=-k^2 c_P\phi.
\ee 
It follows that $\partial_\theta^{-2}\phi_P$ is a generalized eigenfunction for the adjoint problem associated with \eqref{e:gen_ev_bloch} at $\xi=0$
sitting above the genuine $\lambda=0$ eigenfunction $\partial_\theta^{-1}\phi$.  
Combining the above observations with the Fredholm alternative yields the following characterization of the generalized right and left kernels 
for the co-periodic (i.e. $\xi=0$) generalized eigenvalue problem \eqref{e:gen_ev_bloch}.

\begin{theorem}[Co-Periodic Generalized Kernel]
\label{t:co-per}
Assume that Assumption \ref{a:soln} holds, and let $\phi=\phi(\cdot;k,P)$ be a $1$-periodic traveling wave solution of the Ostrovsky equation \eqref{e:ostrov_scale}.
Additionally, assume that Assumption \ref{a:ker} holds at $\phi$.  Then $\lambda=0$ is an $L^2_{\rm per}(0,1)$-eigenvalue associated with the operator pencil eigenvalue 
problem \eqref{e:gen_ev_bloch} at $\xi=0$ with algebraic multiplicity two and geometric multiplicity one.  In particular, defining the functions
\begin{align*}
\Phi_1^0 &= \phi', \hspace{-30mm} &\Phi_2^0 &= \phi_P, \\
\Psi_1^0 &= \d_{\theta}^{-2}\phi_P, \hspace{-30mm}  &\Psi_2^0 &= -\partial_\theta^{-1}\phi
\end{align*} 
it follows that the sets $\{\Phi_{\ell}^0\}_{\ell=1}^2$ and $\{\Psi_j^0\}_{j=1}^2$ provide bases of solutions for the $1$-periodic generalized kernels of 
the operator pencil \eqref{e:gen_ev_bloch} at $\xi=0$, as well as its adjoint problem, respectively.  
Further, the normalization condition
\begin{equation}\label{e:norm1}
\left< \Psi_j^0,\partial_\theta \Phi_{\ell}^0 \right> = \delta_{\ell j}
\end{equation}
holds and the functions $\Phi_{\ell}^0$ and $\Psi_j^0$ satisfy the equations 
\[
\begin{aligned}
&\mathcal{L}_0[\phi]\Phi_1^0 =0,     && \mathcal{L}_0[\phi]\Phi_2^0=-k^2 c_P\partial_\theta\Phi_1^0,\\
&\mathcal{L}_0^\dag[\phi]\Psi_1^0=k^2 c_P\partial_\theta\Psi_2^0,    && \mathcal{L}_0^\dag[\phi]\Psi_2^0=0.
\end{aligned}
\]
\end{theorem}

\begin{proof}
By Assumption \ref{a:ker}, we know that the kernel of $\mathcal{L}_0[\phi]$ is one-dimensional and is spanned by $\phi'$.  Further,
recalling the relation \eqref{e:L_adj_identity} it follows again by Assumption \ref{a:ker} that the kernel of $\mathcal{L}_0^\dag[\phi]$
is spanned by $\partial_\theta^{-1}\phi$.  
Noting that the functions $\phi'$ and $\phi_P$ are necessarily linearly independent, it follows
from \eqref{e:Lident} that $c_P\neq 0$, and hence that $\phi_P$ is in the generalized kernel of $\mathcal{L}_0[\phi]-\lambda k\partial_\theta$ at $\lambda=0$.
Similarly, the above considerations show that the kernel of $\mathcal{L}_0^\dag[\phi]$
is spanned by $\partial_\theta^{-1}\phi$ and that $\partial_\theta^{-2}\phi_P$ lies in the generalized kernel of $\mathcal{L}_0^\dag[\phi]+\lambda k\partial_\theta$
at $\lambda=0$.

To see that $\lambda=0$ is an eigenvalue with algebraic multiplicity exactly two, we note by the Fredholm alternative that the equation
\[
\mathcal{L}_0[\phi]w=\partial_\theta \phi_P
\]
has no $1$-periodic solution $w$ since 
\[
\left<\Psi_2^0,\partial_\theta\phi_P\right>= \left<\phi,\phi_P\right>=\partial_PP=1.
\]
It follows that the generalized kernel of $\mathcal{L}_0[\phi]-\lambda k\partial_\theta$ at $\lambda=0$ has algebraic multiplicity
two and geometric multiplicity one.
The rest of the claims follow similarly and by direct calculations.
\end{proof}

Before continuing, we make two important remarks.  First, we note that the algebraic and geometric multiplicity of $\lambda=0$ as a $1$-periodic
eigenvalue of \eqref{e:gen_ev_bloch} at $\xi=0$ has been rigorously verified for the (asymptotically) small Stokes wave
solutions of the Ostrovsky equation: see \cite{BJP24}.  There, it is rigorously shown that the generalized kernel of the operator $\mathcal{L}_0[\phi]-\lambda k\partial_\theta$
at $\lambda=0$ has algebraic multiplicity two and geometric multiplicity one, which is consistent with Theorem \ref{t:co-per}, and (small amplitude) asymptotic
expansions for the basis elements of the generalized kernel was provided.

Secondly, we note that while the operator $\partial_\theta$ is not invertible on $L^2_{\rm per}(0,1)$, thus requiring us to work with 
the linear operator pencil formulation \eqref{e:gen_ev_bloch}, instead of a standard eigenvalue problem (see Remark \ref{r:no_invert}),
the normalization condition \eqref{e:norm1} implies that $\partial_\theta$ is indeed invertible on the total $1$-periodic
generalized kernel of the operator $\mathcal{L}_0[\phi]-\lambda k\partial_\theta$ at $\lambda=0$.  That is, $\partial_\theta$
is indeed invertible on the total $\lambda=0$ eigenspace associated with \eqref{e:gen_ev_bloch} at $\xi=0$.
This observation will be used heavily in our forthcoming spectral perturbation theory (see, for example, \eqref{e:B_expand} below).
We note, however, that since the bifurcating eigenspace for $|\xi|\ll 1$ is not necessarily mean-zero\footnote{In fact, it is known not to be a mean-zero subspace.}
that even with the above observation the operator $(\partial_\theta+i\xi)^{-1}$ still can not be analytically expanded about $\xi=0$: see, again,
Remark \ref{r:no_invert}.

\subsection{Spectral Perturbation Analysis}
\label{s:perttheory}

Now that we have constructed a basis for the generalized kernels of $\mathcal{L}_0[\phi]-\lambda k\partial_\theta$ at $\lambda=0$,
as well as its adjoint operator at $\lambda=0$, we now consider the spectrum associated to \eqref{e:gen_ev_bloch} 
for $|(\lambda,\xi)| \ll 1$.  To this end, we observe that the Bloch operator $\mathcal{L}_\xi[\phi]$ is analytic in $\xi$ and can
be directly expanded as
\begin{align*}
\mathcal{L}_{\xi}[\phi] &= \gamma + k^2(\d_{\theta}+i\xi)^2(c-\phi-\beta k^2(\d_{\theta}+i\xi)^2) \\
&= L_0 + (ik\xi)L_1 + (ik\xi)^2L_2 + \mathcal{O}(\xi^3),
\end{align*}
where
\be
\label{e:Ls}
L_0= \mathcal{L}_0[\phi], \quad L_1= 2k\d_{\theta}(c-\phi) - 4\beta k^3\d_{\theta}^3, \quad L_2= (c-\phi) - 6\beta k^2\d_{\theta}^2.
\ee 
With the above analytic expression in hand, we recall by Theorem~\ref{t:co-per} that $\lambda=0$ is a $1$-periodic eigenvalue 
for the linear operator pencil $\mathcal{L}_0[\phi]-\lambda k\partial_\theta$ with algebraic multiplicity two and geometric multiplicity one. 
By spectral perturbation theory for linear operator pencils (see Lemma \ref{l:pencil_spec} in Appendix \ref{s:pencil}), we know that
 \eqref{e:gen_ev_bloch} will have two eigenvalues $\{\lambda_j(\xi)\}_{j=1}^2$ near $\lambda=0$ for $|\xi|\ll 1$.
These eigenvalues necessarily bifurcate from the $\lambda_j(0)=0$ state, and the modulational stability, or instability, 
of the underlying wave $\phi$ is thus determined by tracking these two eigenvalues for $|\xi|\ll 1$ and determining if they intersect with the open right half-plane. 
This will be accomplished by projecting the infinite dimensional (operator pencil) eigenvalue problem \eqref{e:gen_ev_bloch}
for $|\xi|\ll 1$ onto the total two-dimensional eigenspace associated with the eigenvalues $\{\lambda_j(\xi)\}_{j=1}^2$ bifurcating from the $(\lambda,\xi)$ state.

By the spectral perturbation theory for linear operator pencils (see Appendix \ref{s:pencil}), along with a standard Graham-Schmidt argument,
we know that the 
dual bases in Theorem~\ref{t:co-per} extend analytically for $|\xi|\ll 1$ into dual right and left bases $\{\Phi_{\ell}^\xi\}_{\ell=1}^2$ and $\{\Psi_j^\xi\}_{j=1}^2$ 
associated to the two eigenvalues $\{\lambda_j(\xi)\}_{j=1}^2$ near the origin and, further, 
that these extended  dual bases satisfy the normalization condition 
\begin{equation}\label{e:normalization}
\left< \Psi_{j}^{\xi},\partial_\theta\Phi_{\ell}^{\xi}\right> = \delta_{\ell j}
\end{equation}
for all $|\xi|\ll 1$. Given the existence of these bases, for $|\xi|\ll 1$ we can construct $\xi$-dependent rank two eigenprojections, namely,
\begin{equation}\label{e:proj}
\left\{
\begin{aligned}
&\Pi(\xi): L^2_{\rm per}(0,1) \to \bigoplus_{j=1}^2 {\rm gker}\Big(\mathcal{L}_{\xi}[\phi]-\lambda_j(\xi)k(\partial_\theta+i\xi)\Big)\\
&\widetilde{\Pi}(\xi): L^2_{\rm per}(0,1) \to \bigoplus_{j=1}^2 {\rm gker}\Big(\mathcal{L}_{\xi}^{\dagger}[\phi]+\overline{\lambda_j(\xi)}k(\partial_\theta+i\xi)\Big)
\end{aligned}\right.
\end{equation}
whose ranges are precisely the total right and left eigenspaces associated with the eigenvalues $\{\lambda_j(\xi)\}_{j=1}^2$ near the origin.  
For details on the construction
(via the Cauchy integral formula / Riesz projector)
of these eigenprojections for the linear operator pencil $\mathcal{L}_\xi[\phi]-\lambda k(\partial_\theta+i\xi)$, see again Lemma \ref{l:pencil_spec} in Appendix \ref{s:pencil}.
In particular, we note that the small eigenvalues $\lambda_j(\xi)$ are given precisely by the roots of the characteristic polynomial
\begin{equation}\label{e:charpoly1}
\det(\mathcal{D}_{\xi} - \lambda\mathcal{B}_\xi) = 0
\end{equation}
associated with the $2\times 2$ matrices
\be
\label{e:matrix_D}
\mathcal{D}_{\xi}:= \widetilde{\Pi}(\xi)\mathcal{L}_{\xi}[\phi]\Pi(\xi) = \left[ \left< \Psi_j^{\xi}, \mathcal{L}_{\xi}[\phi]\Phi_{\ell}^{\xi} \right> \right]_{j,\ell=1}^2
\ee
and
\be\label{e:matrix_B}
\mathcal{B}_{\xi}:= \widetilde{\Pi}(\xi)\left(k(\partial_\theta+i\xi)\right)\Pi(\xi), 
\ee 
defined for $|\xi|\ll 1$, where we have used the normalization \eqref{e:normalization} to simplify the expression for $\mathcal{B}_\xi$.  
It remains to calculate  expansions of the matrices $\mathcal{D}_{\xi}$ and $\mathcal{B}_\xi$ to sufficiently high order.

To this end, we begin by expanding the right and left dual bases as
\begin{align*}
\Phi_{\ell}^{\xi} &= \Phi_{\ell}^0 + (ik\xi)\left(\frac{1}{ik}\d_{\xi}\Phi_{\ell}^{\xi}\Big|_{\xi=0}\right) +\frac{(ik\xi)^2}{2}\left(\frac{1}{(ik)^2}\d_{\xi}^2\Phi_{\ell}^{\xi}\Big|_{\xi=0}\right) + \mathcal{O}(\xi^3), \\
\Psi_{j}^{\xi} &= \Psi_{j}^0 + (ik\xi)\left(\frac{1}{ik}\d_{\xi}\Psi_{j}^{\xi}\Big|_{\xi=0}\right) + \frac{(ik\xi)^2}{2}\left(\frac{1}{(ik)^2}\d_{\xi}^2\Psi_{j}^{\xi}\Big|_{\xi=0}\right) + \mathcal{O}(\xi^3),
\end{align*}
thus yielding the asymptotic expansion 
\begin{equation}\label{e:D_expand}
\mathcal{D}_{\xi} = D_0 + (ik\xi)D_1 + (ik\xi)^2D_2 + \mathcal{O}(\xi^3)
\end{equation}
for $|\xi|\ll 1$.  We first aim to determine the leading order asymptotics for each entry of the $2\times 2$ matrix $\mathcal{D}_{\xi}$ in the limit $\xi\to 0$.
Note that Theorem~\ref{t:co-per} immediately implies that
\be
\label{e:D0}
D_0= \begin{pmatrix} 0 & -k^2c_P \\ 0 & 0 \end{pmatrix}.
\ee
Further, an elementary calculation shows that
\be
\label{e:D_pre}
D_1 = \left[\left< \Psi_j^0, L_0\left(\frac{1}{ik}\d_{\xi}\Phi_{\ell}^{\xi}\Big|_{\xi=0}\right) + L_1\Phi_{\ell}^0 \right> + \left< \left(\frac{1}{ik}\d_{\xi}\Psi_j^{\xi}\Big|_{\xi=0}\right), L_0\Phi_{\ell}^0 \right> \right]_{j,\ell =1}^2.
\ee 
Since \eqref{e:D0} implies that the $(1,2)$ entry of $\mathcal{D}_\xi$ is $\mathcal{O}(1)$, we must calculate
the $(1,1)$, $(2,1)$, and $(2,2)$ entries of $D_1$.  To this end, we make two important observations. First, we note that the normalization of the right and left bases imply that
\[
\d_{\xi}\left< \Psi_j^{\xi}, \partial_\theta\Phi_{\ell}^{\xi} \right> \Big|_{\xi=0} = 0
\] 
and hence,
\be
\label{e:mv_der}
\left< \frac{1}{ik}\d_{\xi}\Psi_j^{\xi}\Big|_{\xi=0}, \partial_\theta\Phi_{\ell}^0\right> = -\left< \Psi_j^0,\partial_\theta \left(\frac{1}{ik}\d_{\xi}\Phi_{\ell}^{\xi}\Big|_{\xi=0}\right) \right>. 
\ee 
Secondly, we note that since $\Phi_{1}^{\xi}$ is in the total right kernel of $\mathcal{L}_{\xi}[\phi]-\lambda k(\partial_\theta+i\xi)$ 
associated with the eigenvalues $\{\lambda_j(\xi)\}_{j=1}^2$ near the origin, it follows that the function
\[
\left(\mathcal{L}_{\xi}[\phi]-\lambda k(\partial_\theta+i\xi)\right)\Phi_1^{\xi}
\]
is invariant with respect to the projection $\Pi(\xi)$ defined in \eqref{e:proj}. Specifically, we have 
\[
\Pi(\xi)\left(\mathcal{L}_{\xi}[\phi]-\lambda k(\partial_\theta+i\xi)\right)\Phi_1^{\xi} =\left(\mathcal{L}_{\xi}[\phi]-\lambda k(\partial_\theta+i\xi)\right)\Phi_1^{\xi}
\]
for all $|\xi|\ll 1$, and hence differentiating with respect to $\xi$ and evaluating at $(\lambda,\xi)=(0,0)$ yields
\be
\label{e:diff_id}
\frac{1}{ik}\d_{\xi}\Pi(\xi)\Big|_{\xi=0}L_0\Phi_1^0 + \Pi(0)\left(L_1\Phi_1^0 + L_0\left(\frac{1}{ik}\d_{\xi}\Phi_1^{\xi}\Big|_{\xi=0}\right)\right) = L_1\Phi_1^0 + L_0\left(\frac{1}{ik}\d_{\xi}\Phi_1^{\xi}\Big|_{\xi=0}\right).
\ee 
Next, differentiating the profile equation \eqref{e:profile_scale} with respect to $k$ gives
\be
\label{e:L_phik}
L_0\phi_k = -k^2c_k\phi'' - L_1\phi'.
\ee
Hence, recalling $L_0\Phi_1^0=0$ by Theorem~\ref{t:co-per} and using \eqref{e:L_phik} it follows that \eqref{e:diff_id} can be written as
\[
\Pi(0)\left(L_0\left(\frac{1}{ik}\d_{\xi}\Phi_1^{\xi}\Big|_{\xi=0}-\phi_k\right)-k^2c_k\phi''\right) = L_0\left(\frac{1}{ik}\d_{\xi}\Phi_1^{\xi}\Big|_{\xi=0}-\phi_k\right)-k^2c_k\phi''.
\] 
Recalling that $L_0\phi_P=-k^2 c_P\phi''$, the above can be rewritten as
\[
\Pi(0)\left(L_0\left(\frac{1}{ik}\d_{\xi}\Phi_1^{\xi}\Big|_{\xi=0}-\phi_k+\frac{c_k}{c_P}\phi_P\right)\right) = L_0\left(\frac{1}{ik}\d_{\xi}\Phi_1^{\xi}\Big|_{\xi=0}-\phi_k
+\frac{c_k}{c_P}\phi_P\right),
\] 
and hence that the function
\[
\frac{1}{ik}\d_{\xi}\Phi_1^{\xi}\Big|_{\xi=0} - \phi_k + \frac{c_k}{c_P}\phi_P
\] 
necessarily lies in the generalized kernel for $\mathcal{L}_\xi[\phi]-\lambda k(\partial_\theta+i\xi)$ at $(\lambda,\xi)=(0,0)$.  
By Theorem~\ref{t:co-per} it follows that
\[
\frac{1}{ik}\d_{\xi}\Phi_1^{\xi}\Big|_{\xi=0} = \phi_k + \sum_{j=1}^2 a_j\Phi_j^0
\]
for some constants $a_j\in\C$. In particular, by replacing the function $\Phi_1^{\xi}$ with the function
\be
\label{e:mod_basis1}
\widetilde{\Phi}_{1}^{\xi} =  \Phi_{1}^{\xi} - (ik\xi)\sum_{j=1}^2a_j\Phi_j^0,
\ee
while simultaneously replacing the left bases $\{\Psi_j^{\xi}\}_{j=1}^2$ with the functions 
\[
\widetilde{\Psi}_j^{\xi} = \Psi_j^{\xi} + (ik\xi)a_j\Psi_1^{\xi}, \quad j=1,2,
\]
it immediately follows that
\[
\widetilde{\Phi}_1^{\xi} = \Phi_1^0 + (ik\xi)\phi_k + \mathcal{O}(\xi^2), \quad \left<\widetilde{\Psi}_j^\xi,\partial_\theta\widetilde{\Phi}_\ell^\xi\right> = I + \mathcal{O}(\xi^2).
\]
Consequently, we have that up to a harmless modification of the right and left basis functions, we can identify the first variation of $\Phi_1^{\xi}$ in $\xi$ at $\xi=0$ as precisely $\phi_k$, while still retaining the normalization \eqref{e:normalization} of the basis up to $\mathcal{O}(\xi^2)$. 
Further, the identity \eqref{e:mv_der} still holds as it only involves first order information in $\xi$. Moving forward,
we will utilize this modified basis while omitting the tilde for ease of notation.

We can now compute the relevant entries of the matrix $D_1$. First, note from \eqref{e:D_pre} that the $(1,1)$ entry is given by
\[
\left< \Psi_1^0, L_0\left(\frac{1}{ik}\d_{\xi}\Phi_1^{\xi}\Big|_{\xi=0}\right) + L_1\Phi_1^0 \right> + \left< \frac{1}{ik}\d_{\xi}\Psi_1^{\xi}\Big|_{\xi=0}, L_0\Phi_1^0 \right>.
\]
Using \eqref{e:mod_basis1} as well as the identity \eqref{e:L_phik} we have
\begin{align*}
\left< \Psi_1^0, L_0\left(\frac{1}{ik}\d_{\xi}\Phi_1^{\xi}\Big|_{\xi=0}\right) + L_1\Phi_1^0 \right> &= \left< \Psi_1^0, L_0\phi_k + L_1\phi' \right> \\
&=-k^2c_k\left< \Psi_1^0, \partial_\theta\Phi_1^0 \right>, \\
&=-k^2c_k. 
\end{align*}
Hence, recalling $L_0\Phi_1^0=0$, it follows that
\[
\left< \Psi_1^0, L_0\left(\frac{1}{ik}\d_{\xi}\Phi_1^{\xi}\Big|_{\xi=0}\right) + L_1\Phi_1^0 \right> + \left< \frac{1}{ik}\d_{\xi}\Psi_1^{\xi}\Big|_{\xi=0}, L_0\Phi_1^0 \right> = -k^2c_k.
\]
Next, the $(2,1)$ entry is given by 
\begin{align*}
&\left< \Psi_2^0, L_0\left(\frac{1}{ik}\d_{\xi}\Phi_1^{\xi}\Big|_{\xi=0}\right) + L_1\Phi_1^0 \right> + \left< \frac{1}{ik}\d_{\xi}\Psi_2^{\xi}\Big|_{\xi=0}, L_0\Phi_1^0 \right>\\
&\qquad\qquad\qquad\qquad\qquad\qquad\qquad\qquad\qquad\qquad = 
	\left< \Psi_2^0, L_0\phi_k + L_1\Phi_1^0 \right> \\
&\qquad\qquad\qquad\qquad\qquad\qquad\qquad\qquad\qquad\qquad= \left< \Psi_2^0, -k^2c_k\partial_\theta\Phi_1^0 \right> \\
&\qquad\qquad\qquad\qquad\qquad\qquad\qquad\qquad\qquad\qquad= -k^2c_k \left< \Psi_2^0, \partial_\theta\Phi_1^0 \right> \\
&\qquad\qquad\qquad\qquad\qquad\qquad\qquad\qquad\qquad\qquad= 0,
\end{align*}
where we have again used the normalization \eqref{e:normalization}.
Finally, for the $(2,2)$ entry we first observe from Theorem \ref{t:co-per} and the identity \eqref{e:mv_der} that
\begin{align*}
\left< \frac{1}{ik}\d_{\xi}\Psi_2^{\xi}\Big|_{\xi=0}, L_0\Phi_2^0 \right> 
&= -k^2c_P \left< \frac{1}{ik}\d_{\xi}\Psi_2^{\xi}\Big|_{\xi=0}, \partial_\theta\Phi_1^0 \right> \\
&= k^2c_P \left< \Psi_2^0, \partial_\theta\frac{1}{ik}\d_{\xi}\Phi_1^{\xi} \right> \\
&= -k^2c_P \left< \partial_\theta^{-1}\phi,\partial_\theta \phi_k \right> \\
&= k^2c_P \left< \phi, \phi_k \right> \\
&= 0,
\end{align*}
where here we have used that $\left< \phi, \phi_k \right> = \partial_kP=0$.  It follows that the $(2,2)$ entry of $D_1$ is given by
\begin{align*}
\left< \Psi_2^0, L_0\left(\frac{1}{ik}\d_{\xi}\Phi_2^{\xi}\Big|_{\xi=0}\right) + L_1\Phi_2^0 \right>  &= \left< L_0^{\dagger}\Psi_2^0, \frac{1}{ik}\d_{\xi}\Phi_2^{\xi}\Big|_{\xi=0} \right> 
	+ \left< \Psi_2^0, L_1\Phi_2^0\right> \\
&= \left< \phi, L_1\phi_P \right>,
\end{align*}
where we used the fact that $L_0^{\dagger}\Psi_2^0=0$ (see Theorem \ref{t:co-per}).
Putting everything together it follows that\footnote{Moving forward, matrix entries marked with $*$ are considered irrelevant,
owing to the fact that lower order terms in the corresponding entry of the expansion of $\mathcal{D}_\xi$ in \eqref{e:D_expand} have already been computed.}
\be
\label{e:D1}
D_1= \begin{pmatrix} -k^2c_k & * \\ 0 & \left< \Psi_2^0,L_1\Phi_2^0 \right> \end{pmatrix}.
\ee

It remains to calculate the $(2,1)$ entry of the matrix $D_2$ which, by an elementary calculation, is given by
\begin{multline*} 
\left< \Psi_2^0, L_0\left(\frac{1}{2(ik)^2}\d_{\xi}^2\Phi_{1}^{\xi}\Big|_{\xi=0}\right) + L_1\left(\frac{1}{ik}\d_{\xi}\Phi_{1}^{\xi}\Big|_{\xi=0}\right) + L_2\Phi_{1}^0 \right> \\
+\left< \frac{1}{ik}\d_{\xi}\Psi_{2}^{\xi}\Big|_{\xi=0}, L_0\left(\frac{1}{ik}\d_{\xi}\Phi_{1}^{\xi}\Big|_{\xi=0}\right) + L_1\Phi_{1}^0 \right> + \left< \frac{1}{2(ik)^2}\d_{\xi}^2\Psi_2^{\xi}\Big|_{\xi=0}, L_0\Phi_{1}^0 \right>.
\end{multline*} 
Noting that $L_0^\dag\Psi_2^0=0$ and $L_0\Phi_1^0=0$, we first observe that
\[
\left< \Psi_2^0, L_0\left(\frac{1}{2(ik)^2}\d_{\xi}^2\Phi_{1}^{\xi}\Big|_{\xi=0}\right) + L_1\left(\frac{1}{ik}\d_{\xi}\Phi_{1}^{\xi}\Big|_{\xi=0}\right) + L_2\Phi_{1}^0 \right> 
	= \left< \Psi_2^0, L_1\phi_k + L_2\phi' \right>, 
\]
and
\[
\left< \frac{1}{(ik)^2}\d_{\xi}^2\Psi_j^{\xi}\Big|_{\xi=0}, L_0\Phi_1^0 \right> = 0.
\]
Finally,
\begin{align*}
\left< \frac{1}{ik}\d_{\xi}\Psi_2^{\xi}\Big|_{\xi=0}, L_0\left(\frac{1}{ik}\d_{\xi}\Phi_1^{\xi}\Big|_{\xi=0}\right) + L_1\Phi_1^0 \right> 
	&= \left< \frac{1}{ik}\d_{\xi}\Psi_2^{\xi}\Big|_{\xi=0}, L_0\phi_k + L_1\phi' \right> \\
&= -k^2c_k \left< \frac{1}{ik}\d_{\xi}\Psi_2^{\xi}\Big|_{\xi=0}, \partial_\theta\Phi_1^0 \right> \\
&= k^2c_k \left< \Psi_2^0, \partial_\theta\frac{1}{ik}\d_{\xi}\Phi_1^{\xi}\Big|_{\xi=0} \right> \\
&= k^2c_k \left< \phi, \phi_k \right> \\
&= 0.
\end{align*}
It follows that
\be
\label{e:D2}
D_2= \begin{pmatrix} * & * \\ \left< \Psi_2^0, L_1\phi_k + L_2\Phi_1^0 \right> & * \end{pmatrix}. 
\ee

Putting \eqref{e:D0}, \eqref{e:D1} and \eqref{e:D2} together, it follows that the matrix $\mathcal{D}_{\xi}$ in \eqref{e:matrix_D} expands analytically for $|\xi|\ll 1$ 
as
\[
\mathcal{D}_\xi=
\left(\begin{array}{cc}
-(ik\xi) k^2 c_k + \mathcal{O}(\xi^2) & -k^2 c_P + \mathcal{O}(\xi)\\
(ik\xi)^2\left< \Psi_2^0, L_1\phi_k + L_2\Phi_1^0 \right> + \mathcal{O}(\xi^3) & (ik\xi)\left<\Psi_2^0,L_1\Phi_2^0\right> + \mathcal{O}(\xi^2)
\end{array}\right).
\]
It remains now to calculate the leading order asymptotics for each entry of the $2\times 2$ matrix $\mathcal{B}_\xi$, defined in \eqref{e:matrix_B}, in the
limit as $\xi\to 0$.  To this end, the normalization condition \eqref{e:normalization} allows us to rewrite $\mathcal{B}_\xi$ as
\begin{equation}\label{e:B_expand}
\mathcal{B}_\xi=kI+(ik\xi) \left[ \left< \Psi_j^{\xi}, \Phi_{\ell}^{\xi} \right> \right]_{j,\ell=1}^2
\end{equation}
Noting that $\left<\Psi_2^0,\Phi_1^0\right> = -\left<\partial_\theta^{-1}\phi,\phi'\right>=\left<\phi,\phi\right>\neq 0$, it follows that\footnote{As we will see,
the $\mathcal{O}(\xi)$ term in the $(1,2)$ entry will not enter into our forthcoming calculation, and hence we will not attempt to calculate it here.}
\[
\mathcal{B}_\xi=\left(\begin{array}{cc}
							k + \mathcal{O}(\xi) & (ik\xi) \left<\Psi_1^0,\Phi_2^0\right> + \mathcal{O}(\xi^2)\\
							(ik\xi)\left<\phi,\phi\right> + \mathcal{O}(\xi^2)& k + \mathcal{O}(\xi)
							\end{array}\right)
\]
Importantly, we note that $\mathcal{B}_\xi$ is necessarily invertible for $|\xi|\ll 1$ with $\det\left(\mathcal{B}_\xi\right)=k^2+\mathcal{O}(\xi^2)$.
It follows from \eqref{e:charpoly1} that the small eigenvalues $\{\lambda_j(\xi)\}_{j=1}^2$ bifurcating from the $(\lambda,\xi)=(0,0)$ state
are given precisely by the eigenvalues of the $2\times 2$ matrix
\begin{align*}
\mathcal{B}_\xi^{-1}\mathcal{D}_\xi &= 
	\left(\begin{array}{cc}
			-(ik\xi)kc_k & -kc_P \\
			(ik\xi)^2\left(\left<\phi,\phi\right>c_k + \frac{1}{k}\left<\Psi_2^0,L_1\phi_k + L_2\Phi_1^0\right>\right) & 
									(ik\xi)\left(\left<\phi,\phi\right>c_P + \frac{1}{k}\left<\Psi_2^0,L_1\Phi_2^0\right>\right) 
			\end{array}\right)\\
&\qquad+
			\left(\begin{array}{cc}
				 \mathcal{O}(\xi^2) &  \mathcal{O}(\xi)\\
				  \mathcal{O}(\xi^3) & \mathcal{O}(\xi^2)
				  \end{array}\right),
\end{align*}
defined for $|\xi|\ll 1$.

By standard spectral perturbation theory (again, see \cite{K76}), it follows that the eigenvalues $\lambda_j(\xi)$ of $\mathcal{B}_\xi^{-1}\mathcal{D}_{\xi}$ 
are at least $C^1$ in $\xi$ for $|\xi|\ll 1$, and satisfy $\lambda_j(0)=0$.  As such, the eigenvalues $\lambda_j$ can be factored as
\be
\label{e:ev_factor}
\lambda_j(\xi) = ik\xi \mu_j(\xi)
\ee
for some continuous functions $\mu_j$ defined for $|\xi|\ll 1$ and satisfying $\mu_j(0)=0$.  Further, we note that by defining the invertible matrix
\be
\label{e:Smat}
S(\xi):= \begin{pmatrix} ik\xi & 0 \\ 0 & 1 \end{pmatrix}
\ee
and defining
\[
\widehat{\mathcal{M}}_{\xi} := S(\xi)\mathcal{B}_\xi^{-1}\mathcal{D}_\xi S(\xi)^{-1},
\]
it follows that $\widehat{\mathcal{M}}_{\xi}$ is analytic in $\xi$ for $|\xi|\ll 1$ and expands as
\begin{align*}
\widehat{\mathcal{M}}_{\xi} &= 
(ik\xi)\left(\begin{array}{cc}
			-kc_k & -kc_P \\
			\left<\phi,\phi\right>c_k + \frac{1}{k}\left<\Psi_2^0,L_1\phi_k + L_2\Phi_1^0\right> & 
									\left<\phi,\phi\right>c_P + \frac{1}{k}\left<\Psi_2^0,L_1\Phi_2^0\right> 
			\end{array}\right)
			+ \mathcal{O}(\xi^2)
\end{align*}
In particular, noting that
\[
\det\left(\mathcal{B}_\xi^{-1}\mathcal{D}_{\xi} - \lambda_j(\xi)I\right) = (ik\xi)^2\det\left(\frac{1}{ik\xi}\widehat{\mathcal{M}}_{\xi}-\mu_j(\xi)I\right)
\]
it follows that the $\mu_j(\xi)$ in \eqref{e:ev_factor} are precisely the eigenvalues of the matrix $\frac{1}{ik\xi}\widehat{\mathcal{M}}_{\xi}$. This establishes the following result.
\begin{theorem}
\label{t:pert}
Under the hypotheses of Theorem~\ref{t:co-per}, the spectrum associated with the (operator pencil) spectral problem \eqref{e:gen_ev1}
considered on $L^2(\R)$ in a sufficiently small neighborhood of the origin consists of precisely two $C^1$ curves $\{\lambda_j(\xi)\}_{j=1}^2$ defined for $|\xi|\ll 1$ which can be expanded as 
\[
\lambda_j(\xi) = ik\xi\mu_j(0) + o(\xi), \quad j=1,2
\]
where the $\mu_j(0)$ are the eigenvalues of the matrix 
\begin{equation}\label{e:M0}
\mathcal{M}_0(\phi) = \left(\begin{array}{cc}
			-kc_k & -kc_P \\
			\left<\phi,\phi\right>c_k + \frac{1}{k}\left<\Psi_2^0,L_1\phi_k + L_2\Phi_1^0\right> & 
									\left<\phi,\phi\right>c_P + \frac{1}{k}\left<\Psi_2^0,L_1\Phi_2^0\right>
			\end{array}\right).
\end{equation}
As such, a necessary condition for the wave $\phi$ to be modulationally stable in the sense of Definition~\ref{d:MI} is that the eigenvalues $\mu_j(0)$ are all real. Furthermore, the wave $\phi$ is modulationally stable provided that the eigenvalues of $\mathcal{M}_0$ are real and distinct. 
\end{theorem}

\begin{proof}
The necessity criteria is clear since if either of the $\mu_j$ has non-zero imaginary parts then the associated eigenvalue $\lambda_j(\xi)$, defined via \eqref{e:ev_factor},
will necessarily have non-zero real part for $0<|\xi|\ll 1$, implying modulational instability.  Further, if both of the $\mu_j(0)$ are real and distinct, then modulational
stability follows from the fact that the eigenvalues associated with \eqref{e:gen_ev_bloch} are necessarily symmetric about
the imaginary axis.  In particular, one would have $\lambda_j(\xi)\in i\RM$ for all $|\xi|\ll 1$ and $j=1,2$, yielding modulational stability.
\end{proof}

The above result provides a mathematically rigorous description of the modulational instability problem for the Ostrovsky equation. As described at the end of Section~\ref{s:Whitham}, however, Whitham's theory of modulations posits that the modulational instability of a given periodic traveling wave $\phi_0$ is determined by the eigenvalues of the matrix 
$\mathcal{W}(\phi_0)$ defined in \eqref{e:Wmat}.  In the next section, we reconcile these two results by proving they are in fact equivalent.

\section{Proof of Theorem~\ref{t:main}}
\label{s:th_proof}

In this section, we establish Theorem~\ref{t:main} by performing a row-by-row computation to show that
\be
\label{e:eqmats}
\mathcal{W}(\phi)=\mathcal{M}_0(\phi)-cI
\ee 
where $\mathcal{W}(\phi)$ is the matrix defined in \eqref{e:Wmat}, i.e. it is associated to the Whitham modulation equations \eqref{e:Wsystem}, 
while the eigenvalues of $\mathcal{M}_0(\phi)$, defined in \eqref{e:M0}, 
rigorously describe the structure of the $L^2(\R)$-spectrum of the (linear operator pencil) spectral problem \eqref{e:gen_ev1}
in a sufficiently small neighborhood of the origin: see Theorem~\ref{t:pert}.  Note that the constant shift by $cI$ does not effect the imaginary parts of the eigenvalues,
and hence does not influence the modulational stability or instability of the underlying wave $\phi$.

First, observe from \eqref{e:Wmat} and \eqref{e:M0} that the first row of the matrices $\mathcal{W}(\phi)$ and $\mathcal{M}_0(\phi) - cI$ are identical. 
To show that the second rows are likewise identical, it is sufficient to establish the identities 
\be
\label{e:row2}
\left\{\begin{aligned}
&\left<\phi,\phi\right>c_k + \frac{1}{k}\left<\Psi_2^0,L_1\phi_k + L_2\Phi_1^0\right>=
	\left< 1, -\frac{1}{3}\phi^3 + \frac{3}{2}\beta k^2\phi_{\theta}^2 + \frac{\gamma}{2k^2}\d_{\theta}^{-2}\phi \right>_k,\\
&\left<\phi,\phi\right>c_P + \frac{1}{k}\left<\Psi_2^0,L_1\Phi_2^0\right>-c=
	\left< 1, -\frac{1}{3}\phi^3 + \frac{3}{2}\beta k^2\phi_{\theta}^2 + \frac{\gamma}{2k^2}\d_{\theta}^{-2}\phi \right>_P.
\end{aligned}\right.
\ee 
We start by studying the second equation above.   To this end, we note that
\begin{align*}
\left<\Psi_2^0,L_1\Phi_2^0\right>&=-\left<\partial_\theta^{-1}\phi,\left(2k\partial_\theta(c-\phi)-4\beta k^3\partial_\theta^3\right)\phi_P\right>\\
&=\left<\phi,2k(c-\phi)\phi_P - 4\beta k^3\phi_P''\right>\\
&=k\left<\phi,2(c-\phi)\phi_P - 4\beta k^2\phi_P''\right>
\end{align*}
and, recalling \eqref{e:Ldagident}, that
\begin{align*}
k^2c_P\left<\phi,\phi\right>&=-\left<L_0^\dag\partial_\theta^{-2}\phi_P,\phi\right>\\
&=-\left<k^2\left(c-\phi-\beta k^2\partial_\theta^2\right)\phi_P + \gamma\partial_\theta^{-2}\phi_P,\phi\right>\\
&=-k^2\left<(c-\phi)\phi_P-\beta k^2\phi_P'' + \frac{\gamma}{k^2}\partial_\theta^{-2}\phi_P,\phi\right>.
\end{align*}
Combining these identities and using integration by parts yields
\begin{align*}
\left<\phi,\phi\right>c_P + \frac{1}{k}\left<\Psi_2^0,L_1\Phi_2^0\right>&=
	\left<\phi,(c-\phi)\phi_P-3\beta k^2\phi_P''-\frac{\gamma}{k^2}\partial_\theta^{-2}\phi_P\right>\\
&=c+\left<1,-\frac{1}{3}\phi^3+\frac{3}{2}\beta k^2 (\phi')^2+\frac{\gamma}{2k^2}(\partial_\theta^{-1}\phi)^2\right>_P,
\end{align*}
which verifies \eqref{e:row2}(ii).

Continuing, we note since $\phi_k$ has mean-zero that the identities \eqref{e:L_adj_identity} and \eqref{e:L_phik} imply that
\[
L_0^\dag\partial_\theta^{-2}\phi_k = \partial_\theta^{-2}L_0\phi_k=-k^2 c_k\phi - \partial_\theta^{-2}L_1\phi'.
\]
Mimicking the above calculations on the momentum terms, we find that
\begin{align*}
k^2 c_k\left<\phi,\phi\right> &= -\left<L_0^\dag\partial_\theta^{-2}\phi_k + \partial_\theta^{-2}L_1\phi',\phi\right>\\
&=-\left<k^2(c-\phi-\beta k^2\partial_\theta^2)\phi_k+\gamma\partial_\theta^{-2}\phi_k+\partial_\theta^{-2}L_1\phi',\phi\right>\\
&=-k^2\left<(c-\phi)\phi_k-\beta k^2\phi_k''+\frac{\gamma}{k^2}\partial_\theta^{-2}\phi_k,\phi\right>-\left<\partial_\theta^{-2}L_1\phi',\phi\right>.
\end{align*}
Similarly, following the above calculations we find
\begin{align*}
\left<\Psi_2^0,L_1\phi_k\right>&=-\left<\partial_\theta^{-1}\phi,\left(2k\partial_\theta(c-\phi)-4\beta k^3\partial_\theta^3\right)\phi_k\right>\\
&=\left<\phi,2k(c-\phi)\phi_k-4\beta k^3\phi_k''\right>\\
&=k\left<\phi,2(c-\phi)\phi_k-4\beta k^2\phi_k''\right>.
\end{align*}
Taken together, the above yields
\begin{align*}
\left<\phi,\phi\right>c_k + \frac{1}{k}\left<\Psi_2^0,L_1\phi_k\right>&=
		\left<\phi,(c-\phi)\phi_k-3\beta k^2 \phi_k''-\frac{\gamma}{k^2}\partial_\theta^{-2}\phi_k\right>
			-\frac{1}{k^2}\left<\partial_\theta^{-2}L_1\phi',\phi\right>\\
&=\left<1,-\frac{1}{3}(\phi^3)_k+\frac{3}{2}\beta k^2\left((\phi')^2\right)_k
	+\frac{\gamma}{2k^2}\left((\partial_\theta^{-1}\phi)^2\right)_k\right>-\frac{1}{k^2}\left<\partial_\theta^{-2}L_1\phi',\phi\right>.
\end{align*}
Comparing to \eqref{e:row2}(i), it remains to calculate 
\begin{align*}
\frac{1}{k}\left<\Psi_2^0,L_2\Phi_1^0\right>-\frac{1}{k^2}\left<\partial_\theta^{-2}L_1\phi',\phi\right>&=
	-\frac{1}{k}\left<\partial_\theta^{-1}\phi,(c-\phi)\phi'-6\beta k^2\phi'''\right>\\
&\qquad\qquad			+\frac{1}{k^2}\left<2k(c-\phi)\phi'-4\beta k^3\phi''',\partial_\theta^{-1}\phi\right>\\
&=\frac{1}{k}\left<\partial_\theta^{-1}\phi,(c-\phi)\phi'+2\beta k^2\phi'''\right>,
\end{align*}
where the first equality follows by integration by parts and the definitions of $L_1$ and $L_2$ in \eqref{e:Ls}.  Now, noting that the equation
$L_0\phi'=0$ implies that
\[
(c-\phi)\phi'=\beta k^2\phi''' - \frac{\gamma}{k^2}\partial_\theta^{-1}\phi
\]
it follows that
\begin{align*}
\frac{1}{k}\left<\Psi_2^0,L_2\Phi_1^0\right>-\frac{1}{k^2}\left<\partial_\theta^{-2}L_1\phi',\phi\right>&=
		\left<\partial_\theta^{-1}\phi,3\beta k\phi'''-\frac{\gamma}{k^3}\partial_\theta^{-1}\phi\right>\\
&=\left<1,3\beta k\left(\phi'\right)^2-\frac{\gamma}{k^3}\left(\partial_\theta^{-1}\phi\right)^2\right>.
\end{align*}
Putting the above calculations together now yields
\begin{align*}
\left<\phi,\phi\right>c_k+\frac{1}{k}\left<\Psi_2^0,L_1\phi_k+L_2\Phi_1^0\right>&=
		\left<1,-\frac{1}{3}(\phi^3)_k+\frac{3}{2}\beta k^2\left((\phi')^2\right)_k
	+\frac{\gamma}{2k^2}\left((\partial_\theta^{-1}\phi)^2\right)_k\right>\\
&\qquad\qquad					\left<1,3\beta k\left(\phi'\right)^2-\frac{\gamma}{k^3}\left(\partial_\theta^{-1}\phi\right)^2\right>	\\
&=	\left<1,-\frac{1}{3}\phi^3+\frac{3}{2}\beta k^2\left(\phi'\right)^2
	+\frac{\gamma}{2k^2}\left(\partial_\theta^{-1}\phi\right)^2\right>_k,
\end{align*}
which verifies \eqref{e:row2}(i).  This establishes the key identity \eqref{e:eqmats}.   

\

The above shows that if the Whitham matrix $\mathcal{W}(\phi)$ has eigenvalues $\{\alpha_j\}_{j=1}^2$ then the linearized spectral 
problem \eqref{e:gen_ev1} has two curves of essential spectra near $\lambda=0$ of the form
\[
\lambda_j(\xi)=i(\alpha_j+c)\xi+o(\xi),~~j=1,2.
\]
Clearly then, if $\mathcal{W}(\phi)$ is elliptic, i.e. if the $\alpha_j$ have non-zero imaginary parts, this immediately implies spectral instability
of the underlying wave $\phi$.  On the other hand, the above only provides information about the leading-order term in the  
expansion of the $\lambda_j$, and hence to complete the proof of Theorem \ref{t:main} we must
show that strict hyperbolicity is sufficient to ensure that $\lambda_j(\xi)\in i\RM$ for all $|\xi|\ll 1$, i.e. that the higher order terms do not contribute
to a spectral instability in this case.

To this end, suppose that the $\alpha_j$ are real and distinct (i.e. that $\mathcal{W}(\phi)$ is strictly hyperbolic).  By the distinctness of the $\alpha_j$, 
spectral perturbation theory implies that the spectral curves are in-fact analytic in $\xi$, and hence
admit expansions of the form
\begin{equation}\label{e:spec_expansion}
\lambda_j(\xi)=i(\alpha_j+c)\xi+\mathcal{O}(\xi^2),~~j=1,2.
\end{equation}
Suppose now that $\lambda_j(\xi)$ has a non-zero real-parts for $0<|\xi|\ll 1$, indicating a spectral
instability near $\lambda=0$.  Owing to the spectral symmetry \eqref{e:spec_sym} from Remark \ref{r:spec_sym}, it follows that we must have
\[
\lambda_1(\xi)=-\overline{\lambda_2(\xi)},~~0<|\xi|\ll 1.
\]
Using the expansion \eqref{e:spec_expansion}, this implies that $\alpha_1=\alpha_2$, thus contradicting strict hyperbolicity of the Whitham matrix 
$\mathcal{W}(\phi)$.  It follows that if $\mathcal{W}(\phi)$ is strictly hyperbolic then the both spectral curves $\lambda_j(\xi)$
must be confined to the imaginary axis for $|\xi|\ll 1$, implying spectral stability in a sufficiently small neighbohrood of the origin in the spectral plane.
This completes the proof of Theorem \ref{t:main}.

\begin{remark}
We emphasize that the above argument that strict hyperbolicity of $\mathcal{W}(\phi)$ implies spectral stability near $\lambda=0$
relies strongly on the spectral symmetry \eqref{e:spec_sym} which, in turn, relies effectively on the Hamiltonian structure of the governing PDE.
In non-Hamiltonian cases (for example, in the presence of dissipation), strict hyperbolicity
is indeed insufficient to conclude spectral stability near $\lambda=0$.  In such cases, one must obtain information on the higher order
terms in the spectral expansion \eqref{e:spec_expansion}, which (in some cases) can be obtained by continuing the WKB expansion
from Section \ref{s:Whitham} to higher order or by numerical Evans function calculations.  See, for instance, \cite{BJNRZ13,BJNRZ17,DSS09,JNRZ14,SSSU12}.
\end{remark}

\section{Comments on the Reduced Ostrovsky Equation}
\label{s:W_reducedostrov}

In all of our above work, we exclusively focused on the classical Ostrovsky equation \eqref{e:ostrov} with $\beta\neq 0$.  In this final section,
we will discuss the application of our theory to the reduced Ostrovsky equation \eqref{e:reduced_ostrov}.  While we do not attempt to provide all
the details here, the main point is that nearly all of the above results still apply in the case of the reduced model with $\beta=0$.  Note this should
not be immediately obvious given that the reduced Ostrovsky model \eqref{e:reduced_ostrov} is clearly, from a PDE standpoint, a singular
perturbation of the classical Ostrovsky equation \eqref{e:ostrov} with $\beta\neq 0$. 

\

To begin, we note that in the traveling coordinates $(\eta,t)=(x-ct,t)$ the reduced Ostrovsky equation \eqref{e:reduced_ostrov} can be rewritten as
\be
\label{e:reduced_ostrov_rewrite}
\left(u_t-cu_\eta+uu_\eta\right)_\eta=\gamma u
\ee
and hence periodic traveling waves $u(\eta,t)=\phi(\eta)$ are found as solutions of the (reduced) profile equation
\begin{equation}\label{e:r_profile}
\left(\left(c-\phi\right)\phi'\right)'+\gamma\phi=0.
\end{equation}
Multiplying the above by $(c-\phi)\phi'$ reduces \eqref{e:r_profile} to the quadrature form
\[
\frac{1}{2}\left(c-\phi\right)^2 \left(\phi'\right)^2=\frac{\gamma}{3}\phi^3-\frac{c}{2}\phi^2+E,
\]
where $E\in\RM$ is a constant of integration.  Using direct phase plane analysis, the authors in \cite{GP17_2} prove
the above produces a $2$-parameter family of $T=T(c,E)$ periodic solutions $\phi(\cdot;c,E)$ of the reduced profile equation \eqref{e:r_profile}
that can be smoothly parameterized by the wave speed $c$ and the ODE energy level $E$.  Further, through a detailed 
analysis of the period and momentum function, the authors prove for each $c>0$ that the mapping $E\to T(c,E)$ is monotonically decreasing
while, for fixed (admissible) $E$), the mappings $c\to T(c,E)$ and $c\to P(c,E)$ are both monotonically increasing.  
By the Implicit Function Theorem, setting $k=1/T$ it follows that for each admissible $(c,E)$ that this family can be smoothly reparameterized
in terms of the frequency $k$ and the wave speed $c$ via  $\phi(\cdot;k,c)$.  

Of course, without information regarding
the monotonicity of the mapping $E\mapsto P(c,E)$ we can not conclude this family can be reparameterized in 
terms of the frequency $k$ and the momentum $P$, i.e. we can not directly verify Assumption \ref{a:soln} in this reduced case.
Nevertheless, throughout the remainder of this section we will continue our discussion under the assumption
that Assumption \ref{a:soln} continues to hold for the reduced profile equation \eqref{e:r_profile}.
Regarding Assumption \ref{a:ker}, we note that this was analytically verified in \cite[Lemma 2.8]{JP16} for the reduced model in the limit of asymptotically
small (Stokes) waves.  Moving forward, we will additionally operate under the assumption that Assumption \ref{a:ker} continues 
to hold in the reduced ($\beta=0$) Ostrovsky equation.

Assuming that both Assumption \ref{a:soln} and Assumption \ref{a:ker} hold for the reduced Ostrovsky equation, we note that 
all of the analysis in Sections \ref{s:Whitham} - \ref{s:th_proof} holds.  Indeed, one can directly check that all of the calculations
from those sections continue to work verbatim.  In particular, performing the multiple scales analysis from Section \ref{s:Whitham}
on the reduced Ostrovsky equation yields 
the first-order $2\times 2$ Whitham system, namely,
\be
\label{e:reduced_Whitham}
\left\{\begin{aligned}
&k_T = -(kc)_X \\
&P_T = \left< 1, -\frac{1}{3}(u^0)^3 + \frac{\gamma}{2k^2}(\d_{\theta}^{-1}u^0)^2 \right>.
\end{aligned}\right.
\ee 
Using the chain rule, the (reduced) Whitham system \eqref{e:reduced_ostrov} can be written in the quasilinear form
\be
\label{e:reduced_quasilin}
\begin{pmatrix} k \\ P \end{pmatrix}_T = \mathcal{W}_{\beta\to 0}(u^0) \begin{pmatrix} k \\ P \end{pmatrix}_X,
\ee
where
\be
\label{e:reduced_Wmat}
\mathcal{W}_{\beta\to0}(u) =
\left(\begin{array}{cc}
-(kc)_k & -(kc)_P
\\
\left<1, -\frac{1}{3}u^3 + \frac{\gamma}{2k^2}(\d_{\theta}^{-1}u)^2\right>_k & \left<1, -\frac{1}{3}u^3 + \frac{\gamma}{2k^2}(\d_{\theta}^{-1}u)^2\right>_P
\end{array}\right).
\ee
Interestingly, we note that the Whitham system for reduced Ostrovsky equation \eqref{e:reduced_ostrov} 
agrees precisely with that obtained by formally setting $\beta=0$ in \eqref{e:Wsystem}.  That is, under Assumption \ref{a:soln} and Assumption \ref{a:ker}, 
the modulation system for the reduced Ostrovsky is a regular perturbation of the modulation system for \eqref{e:ostrov} even though
the reduced Ostrovsky equation is clearly a singular perturbation of the classical Ostrovsky equation \eqref{e:ostrov}.

Similarly, one can directly check that all of the results and calculations in Sections \ref{s:modtime}-\ref{s:th_proof} continue to hold, with their proofs
in the $\beta=0$ case following verbatim the classical $\beta\neq 0$ case.  As such, it follows that our main result Theorem \ref{t:main}
continues to hold for the reduced Ostrovsky equation \ref{e:reduced_ostrov}, i.e. the hyperbolicity/ellipticity of the
(reduced) Whitham matrix \eqref{e:reduced_Wmat} corresponds to the modulational stability/instability (respectively) of the underlying
periodic traveling wave of the reduced model.  Again, this is all under the assumption that Assumption \ref{a:soln} and Assumption \ref{a:ker} continue
to hold in the reduced setting.

Given the above, we observe that due to the known (subharmonic) orbital stability result for the reduced Ostrovsky equation
in \cite{JP16}, it immediately follows that periodic traveling wave solutions for \eqref{e:reduced_ostrov} are necessarily spectrally
stable to square integrable perturbations.  By the theoretical setup in Section \ref{s:modtime}, it follows
that all periodic traveling wave solutions to the reduced Ostrovsky equation are necessarily (spectrally) modulationally stable.
In particular, the above considerations imply then that the Whitham system \eqref{e:reduced_quasilin} is necessarily
always hyperbolic.  While we do not have a direct proof of hyperbolicity outside of the above considerations, 
it would be interesting to use the integrability of the reduced
Ostrovsky equation to verify this claim directly from Whitham's modulation theory.

\appendix

\section{Spectral \&  Perturbation Theory}\label{s:pencil}

In our rigorous spectral analysis in  Section \ref{s:modtime}, we encountered the spectral problem
\begin{equation}\label{e:pencil0}
\mathcal{L}[\phi] v = \lambda \mathcal{H} v
\end{equation}
considered on $L^2(\RM)$, where 
\[
\mathcal{L}[\phi]:=\gamma+k^2\partial_\theta^2\left(c-\phi-\beta k^2\partial_\theta^2\right)
\]
and
\[
\mathcal{H}:= k\partial_\theta
\]
are closed, densely defined linear operators on $L^2(\RM)$ with $1$-periodic coefficients.
While the spectral and spectral perturbation theory for linear operator
pencils is similar to the classical theory for closed, densely defined linear operators (see Kato \cite{K76}), the differences are enough that we feel
it is worthwhile to collect the relevant results here.  The forthcoming discussion is largely based on the presentations in \cite{Baum85,GGK90,JMMP14,MP15}, and readers are 
recommended to consult these references for more details.

We begin with the following observation.

\begin{lemma}\label{L:essential}
The $L^2(\RM)$-spectrum of \eqref{e:pencil0} is entirely essential, i.e. the point spectrum is empty.
\end{lemma}

\begin{proof}
The proof follows closely that in \cite[Lemma 3.3]{JMMP14}.
Suppose, on the contrary, that $\lambda_0\in\sigma_p(\mathcal{L}[\phi],\mathcal{H})$ and note, by definition,
that $\lambda_0\mathcal{H}-\mathcal{L}[\phi]$ is a Fredholm operator with index zero.  In particular, the associated
eigenspace $\Sigma_0={\rm ker}\left(\lambda_0\mathcal{H}-\mathcal{L}[\phi]\right)$ is finite dimensional.  Defining the period-shift operator
\[
\mathcal{S}:L^2(\RM;\mathbb{C})\to L^2(\RM;\mathbb{C}),~~~\mathcal{S}f(\theta)=f(\theta+1)
\]
it follows that $\Sigma_0$ is an invariant subspace for  $\mathcal{S}$ and further, owing to the periodicity of the 
coefficients of $\mathcal{L}[\phi]$, that $\mathcal{S}$ commutes with $\mathcal{L}[\phi]$.  Restricting $\mathcal{S}$
to the finite-dimensional eigenspace $\Sigma_0$ results in a finite-dimensional unitary operator $\widetilde{\mathcal{S}}$
which must therefore have an eigenvalue $e^{i\xi}$ for some $\xi\in[-\pi,\pi)$ with corresponding eigenfunction
$\tilde{v}\in\Sigma_0$.  It then follows that for each $m\in\ZM$ we have
\[
|\tilde{v}(m)|=\left|\widetilde{\mathcal{S}}^m \tilde{v}(0)\right|=\left|e^{im\xi}\tilde{v}(0)\right|=|\tilde{v}(0)|,
\]
which contradicts that $\tilde{v}\in L^2(\RM)$ unless $\tilde{v}=0$.  It follows that the point spectrum is empty, as claimed.
\end{proof}

To analyze the essential spectrum of \eqref{e:pencil0}, we note that while the spectral problem \eqref{e:pencil0} is non-standard,
taking the form of a linear operator pencil, it can be rewritten as a $3\times 3$ system of ODEs of the form
\begin{equation}\label{e:sys_ode}
Y'(\theta) = \mathbb{A}(\theta,\lambda)Y(\theta),~~~Y = (v,v',v'')^T.
\end{equation}
Defining $\mathbb{F}(\theta,\lambda)$ to be the associated fundamental matrix solution, i.e. the unique $3\times 3$ matrix solution of \eqref{e:sys_ode} satisfying
$\mathbb{F}(0,\lambda)=I$ for all $\lambda\in\CM$, the fact that the coefficients of $\mathcal{L}[\phi]$, and hence the matrix $\mathbb{A}$, are $1$-periodic
implies that $\mathbb{F}$ satisfies
\[
\mathbb{F}(\theta+1,\lambda)=\mathbb{F}(\theta,\lambda)\mathbb{M}(\lambda)
\]
for all $\theta\in\RM$, where here $\mathbb{M}(\lambda):=\mathbb{F}(1,\lambda)$ is the monodromy matrix (i.e. period map) associated 
with \eqref{e:sys_ode}.  By standard ODE theory, the monodromy matrix is an entire function of $\lambda\in\CM$ and,
further, that any eigenfunction $\tilde{Y}(\lambda)$ of $\mathbb{M}(\lambda)$ necessarily induces the non-trivial solution 
${ Y}(\theta,\lambda)=\mathbb{F}(\theta,\lambda)\tilde{ Y}(\lambda)$ of \eqref{e:sys_ode} which, furthermore, may
be decomposed as
\[
{ Y}(\theta;\lambda)=e^{\mu(\lambda)\theta}{ W}(\theta;\lambda)
\]
where ${ W}(\theta+1;\lambda)={W}(\theta;\lambda)$ for all $\theta\in\RM$.  The next result characterizes the $L^2(\RM)$-spectrum
of \eqref{e:pencil0} in terms of the monodromy operator $\mathbb{M}(\lambda)$.

\begin{lemma}\label{L:evans}
We have $\lambda\in\sigma_{L^2(\RM)}\left(\mathcal{L}[\phi],\mathcal{H}\right)$ if and only if there exists a $\xi\in[-\pi,\pi)$
such that
\[
\det\left(\mathbb{M}(\lambda)-e^{i\xi}I\right)=0.
\]
\end{lemma}

\begin{proof}
We have already established in Lemma \ref{L:essential} that the spectrum of \eqref{e:pencil0} is purely essential.  Further,
it is well known (see \cite{Gard93,Hen81}) that $\lambda\in\sigma_{\rm ess}\left(\mathcal{L}[\phi],\mathcal{H}\right)$
if and only if the ODE system \eqref{e:sys_ode} or, equivalently, the ODE \eqref{e:pencil0}, admits a non-trivial solution
that is uniformly bounded on $\RM$: see, for example, \cite[pg. 139-140]{Hen81} or \cite[Proposition 2.1]{Gard93}.  By the above
considerations, the existence of such a non-trivial bounded solution will, in turn, 
happen if and only if the monodromy operator $\mathbb{M}(\lambda)$ has an eigenvalue on the unit circle.  
\end{proof}

By the above work, it follows that $\lambda\in\sigma_{L^2(\RM)}\left(\mathcal{L}[\phi],\mathcal{H}\right)$ if and only if the differential equation
\eqref{e:pencil0} admits a non-trivial solution of the form
\[
v(\theta) = e^{i\xi\theta}w(\theta)
\]
for some $\xi\in[-\pi,\pi)$ and some $w\in L^2_{\rm per}(0,1)$.  Substituting into \eqref{e:pencil0} leads to the one-parameter
family of spectral problems
\begin{equation}\label{e:pencil1}
\mathcal{L}_\xi[\phi]w=\lambda\mathcal{H}_\xi w
\end{equation}
considered on $L^2_{\rm per}(0,1)$, where
\[
\mathcal{L}_\xi[\phi]:=e^{-i\xi\theta}\mathcal{L}[\phi]e^{i\xi\theta}=\gamma+k^2(\partial_\theta+i\xi)^2\left(c-\phi-\beta k^2(\partial_\theta+i\xi)^2\right)
\]
and
\[
\mathcal{H}_\xi:= k(\partial_\theta+i\xi)
\]
denote the associated Bloch operators.  
Our next result characterizes the spectrum associated to \eqref{e:pencil1}.

\begin{lemma}\label{L:discrete}
With the above setup, for each fixed $\xi\in[-\pi,\pi)$ we have 
\[
\sigma(\mathcal{L}_\xi[\phi],\mathcal{H}_\xi)=\sigma_p(\mathcal{L}_\xi[\phi],\mathcal{H}_\xi).
\]
That is, the $L^2_{\rm per}(0,1)$-spectrum
associated to the linear operator pencil \eqref{e:pencil1} is comprised entirely of isolated eigenvalues with finite (algebraic) multiplicity.
\end{lemma}

\begin{proof}
For a fixed $\xi\in[-\pi,\pi)$, we know from Lemma \ref{L:evans} and its following discussion that  
$\lambda_0\in\sigma_{L^2_{\rm per}(0,1)}\left(\mathcal{L}_\xi[\phi],\mathcal{H}_\xi\right)$
if and only $\lambda_0$ is a root of the function
\[
\CM\ni\lambda\mapsto D(\lambda,\xi):=\det\left(\mathbb{M}(\lambda)-e^{i\xi}I\right)
\]
and that such a root necessarily induces a non-trivial $1$-periodic solution of \eqref{e:pencil1}.  This shows that
the spectrum of \eqref{e:pencil1} is indeed comprised entirely of eigenvalues.  
Further, since the function $D(\lambda,\xi)$, known in the literature as the periodic Evans function, is known to be an entire function of $\lambda$, we know
that its roots are necessarily isolated (i.e. they have no finite accumulation point) and have finite multiplicity (i.e. if $\lambda_0$ is a root then there exists a $k\in\NM$
such that $\partial_\lambda^kD(\lambda_0,\xi)\neq 0$).  Since the roots of $D(\cdot,\xi)$ correspond in location and (algebraic) multiplicity with the 
eigenvalues of \eqref{e:pencil1} (see \cite{Gard93}) it follows that the point spectrum consists of isolated eigenvalues with finite multiplicity, as claimed.
\end{proof}

\begin{remark} 
By combining Lemma \ref{L:evans} with Lemma \ref{L:discrete}, it follows that we have the spectral decomposition
\[
\sigma_{L^2(\RM)}\left(\mathcal{L}[\phi],\mathcal{H}\right)=\bigcup_{\xi\in[-\pi,\pi)}\sigma_p\left(\mathcal{L}_\xi[\phi],\mathcal{H}_\xi\right).
\]
This provides a continuous parameterization of the essential spectrum of \eqref{e:pencil0} in terms of the $1$-parameter family
of eigenvalues coming from \eqref{e:pencil1}.  Note that, alternative to the above considerations, we could have established the above spectral
decomposition directly by using the functional-analytic Bloch Transform
\[
B:L^2(\RM)\to L^2\left([-\pi,\pi);L^2_{\rm per}(0,1)\right),~~B(v)(\xi,x)=\sum_{\ell\in\ZM} e^{2\pi i\ell x}\hat{v}(\xi+2\pi \ell),
\]
where $\hat{v}(z)=\int_{\RM} e^{-iz\omega}v(\omega)d\omega$ denotes the standard Fourier Transform of $v$,
to prove that $\lambda\mathcal{H}-\mathcal{L}[\phi]$ is boundedly invertible on $L^2(\RM)$ if and only if 
the Bloch operators $\lambda\mathcal{H}_\xi-\mathcal{L}_\xi[\phi]$ are boundedly invertible on $L^2_{\rm per}(0,1)$
for all $\xi\in[-\pi,\pi)$.  See, for instance, the argument in \cite[Appendix B]{J2013} in the context of a fractional KdV equation.
\end{remark}

The above provides a global description of the spectrum associated to the linear operator pencil \eqref{e:pencil0} and, in particular, its relationship
to the $1$-parameter family of eigenvalue problems \eqref{e:pencil1}.  In our work in Section \ref{s:modtime}, we saw that we additionally
need an appropriate spectral perturbation theory for the eigenvalue problems \eqref{e:pencil1}.  To this end,
we note that in  Theorem \ref{t:co-per} we established that $\lambda=0$ is an isolated eigenvalue for the linear operator
pencil \eqref{e:pencil1} at $\xi=0$ with algebraic multiplicity two and geometric multiplicity one.  To track the splitting of this (double) eigenvalue
for $|\xi|\ll 1$, we rely on the spectral perturbation theory for linear operator pencils developed by M\"oller \& Pivovarchik \cite{MP15}.
Specifically, since the coefficients of \eqref{e:pencil1} depend analytically on both $\lambda$ and $\xi$ for $|(\lambda,\xi)|\ll 1$, 
the next result follows immediately by applying Theorem 9.2.4 and Lemma 1.1.9 in \cite{MP15} to the present context, and we refer the reader there for details.

\begin{lemma}\label{l:pencil_spec}
Under the assumptions of Theorem \ref{t:co-per}, there exists a $\delta>0$ such that for all $|\xi|\ll 1$ the linear operator
pencil \eqref{e:pencil1} has exactly two eigenvalues, counted with respect to algebraic multiplicity, interior to the disk $B(0,\delta)\subset\CM$.
In particular, we have the following:
\begin{enumerate}
\item The eigenvalues $\{\lambda_j(\xi)\}_{j=1}^2$ bifurcating from the $(\lambda,\xi)=(0,0)$ state are either analytic in $\xi$ or else
they admit a convergent Puiseux series representation 
\[
\lambda_j(\xi)=\sum_{n=1}^\infty a_n\xi^{n/2}.
\]
Either way, the eigenvalues $\lambda_j(\xi)$ are continuous in $\xi$ for $|\xi|\ll 1$.

\item If $\Gamma=\partial B(0,\delta)$ then
\[
\Pi(\xi):=\frac{1}{2\pi i}\oint_\Gamma\mathcal{H}_\xi\left(\lambda \mathcal{H}_\xi-\mathcal{L}_\xi[\phi]\right)^{-1}d\lambda
\]
is a well-defined projection operator on $L^2_{\rm per}(0,1)$ with rank 2.  Specifically, 
\[
\Pi(\xi):L^2_{\rm per}(0,1)\to\bigoplus_{j=1}^2{\rm gker}\left(\lambda_j(\xi)\mathcal{H}_\xi-\mathcal{L}_\xi[\phi]\right).
\]
is the spectral projection onto the total eigenspace bifurcating form the $(\lambda,\xi)=(0,0)$ state.
\end{enumerate}
\end{lemma}

Key to our spectral perturbation analysis in Section \ref{s:perttheory} was the analytic dependence on the critical bifurcating eigenspaces
from the $(\lambda,\xi)=(0,0)$ state on the Bloch parameter $\xi$.  To justify this, we must establish that the resolvent operator
\[
\left(\lambda\mathcal{H}_\xi-\mathcal{L}_\xi[\phi]\right)^{-1}
\]
is analytic in the two variables $\lambda$ and $\xi$.  Analytic dependence on $\lambda$ follows by standard arguments, while analyticity
with respect to $\xi$ follows due to the fact that the coefficients of $\lambda\mathcal{H}_\xi-\mathcal{L}_\xi[\phi]$ depend analytically
on $\xi$.  More precisely, for $\lambda\in\rho(\mathcal{L}_\xi[\phi],\mathcal{H}_\xi)$ we let $U:H^4_{\rm per}(0,1)\to L^2_{\rm per}(0,1)$ map bijectively
onto $H^4_{\rm per}(0,1)\subset L^2_{\rm per}(0,1)$ via the identity map.  The key is then to note that the operator
\[
V_\xi:H^4_{\rm per}(0,1)\to L^2_{\rm per}(0,1),~~~V_\xi:=\left(\lambda\mathcal{H}_\xi-\mathcal{L}_\xi[\phi]\right)U
\]
is a boundedly analytic operator of $\xi$ for $|\xi|\ll 1$.  Since $V_0$ is boundedly invertible, it follows that $V_\xi$ is boundedly invertible
for $|\xi|\ll 1$ and that
\[
\left(\lambda\mathcal{H}_\xi-\mathcal{L}_\xi[\phi]\right)^{-1}=UV_\xi^{-1}:L^2_{\rm per}(0,1)\to L^2_{\rm per}(0,1),~~|\xi|\ll 1.
\]
In particular, being the composition of two boundedly analytic operators, it follows that the resolvent is indeed analytic in $\xi$ for $|\xi|\ll 1$.
Since $\mathcal{H}_\xi$ clearly depends analytically on $\xi$, it follows from Lemma \ref{l:pencil_spec} 
that the projection operator $\Pi(\xi)$ will depend analytically on $\xi$ for $|\xi|\ll 1$.  In particular, it follows that the 
total eigenspace bifurcating from the $(\lambda,\xi)=(0,0)$ state, corresponding to the range of $\Pi(\xi)$, 
depends analytically on $\xi$ for $|\xi|\ll 1$.

\bibliographystyle{abbrv}
\bibliography{ost}

\end{document}